\documentclass[11pt]{article}%
\usepackage{tocbibind}
\usepackage[toc,page]{appendix}
\usepackage[margin=2cm]{geometry}
\usepackage[utf8]{inputenc}
\usepackage[T1]{fontenc}
\usepackage{lmodern}
\usepackage[french,english]{babel}
\usepackage{amsthm, amssymb, amsmath, amsfonts, mathrsfs, esint, textcomp, bm, xcolor}
\usepackage[colorlinks=true, pdfstartview=FitV, linkcolor=blue, citecolor=blue, urlcolor=blue,pagebackref=false]{hyperref}
\usepackage{mathtools}
\usepackage{tikz}
\usepackage{enumitem, imakeidx}
\usetikzlibrary{patterns}

\usepackage{microtype}

\definecolor{darkblue}{rgb}{0,0,0.7} 
\definecolor{darkred}{rgb}{0.9,0.1,0.1}
\definecolor{darkgreen}{rgb}{0,0.5,0}

\newtheorem{thm}{Theorem}[section]
\newtheorem{prop}[thm]{Proposition}
\newtheorem{lem}[thm]{Lemma}
\newtheorem{cor}[thm]{Corollary}
\theoremstyle{remark}
\newtheorem{rem}[thm]{Remark}
\theoremstyle{definition}

\newcommand{\vv}{\vspace{0.1cm}}

\renewcommand{\O}{\mathcal{O}}

\renewcommand{\leq}{\leqslant}
\renewcommand{\geq}{\geqslant}

\renewcommand{\subset}{\subseteq}

\newcommand{\aloc}{a_{\mathrm{loc}, \eps}}

\newcommand{\N}{\mathbb{N}}

\newcommand{\1}{\mathbf{1}}
\newcommand{\R}{\mathbb{R}}
\newcommand{\C}{\mathbb{C}}

\newcommand{\Z}{\mathbb{Z}}

\newcommand{\eps}{\varepsilon}
\renewcommand{\d}{{\mathrm{d}}}

\newcommand{\T}{\mathbb{T}}

\newcommand{\vT}{\vec{T}}
\newcommand{\vN}{\vec{N}}
\newcommand{\dist}{\mathop{dist}}
\newcommand{\iii}[1]{{\left\vert\kern-0.25ex\left\vert\kern-0.25ex\left\vert #1 
    \right\vert\kern-0.25ex\right\vert\kern-0.25ex\right\vert}}
\newcommand{\Hiwa}{H_{\text{Iwa}}}

\newcommand{\Psif}{\Psi_{\eps, \textrm{flat}}}
\newcommand{\Imm}{\text{Im}}
\newcommand{\Psifbar}{\overline{\Psi}_{\eps, \textrm{flat}} }

\numberwithin{equation}{section}

\begin{document}

\begin{center}
{\Large Edge States for generalized Iwatsuka models: Magnetic fields having a fast transition across a curve}\\
\bigskip
{ARIANNA GIUNTI, JUAN J.L. VEL\'AZQUEZ}\\

\end{center}

\bigskip

\textbf{Abstract:} In this paper, we study the localization and propagation properties of the edge states associated to a class of magnetic laplacians in $\R^2$. We assume that the intensity of the magnetic field has a fast transition along a regular and compact curve $\Gamma$. Our main results extend to a general regular curve the study of the localised eigenfunction obtained when $\Gamma$ is a straight line (i.e. \textit{Iwatsuka models}). Furthermore, we include in our analysis the case of magnetic fields that slowly change along the curve $\Gamma$ and we obtain a rigorous and explicit characterization of the asymptotic mass distribution of the edge state along $\Gamma$.

\medskip


\bigskip

\section{Introduction}
In this paper we study the existence of localized states for a class of magnetic Laplacians
\begin{align}\label{hamiltonian.intro}
H_\eps := -(\nabla + i \eps^{-2}a_\eps) \cdot (\nabla + i \eps^{-2}a_\eps) \ \ \ \ \text{in $\R^2$}
\end{align}
in the semi-classical regime $\eps \ll 1$. Here, the vector potential $a_\eps: \R^2 \to \R^2$ satisfies $\nabla \times a_\eps = b_\eps e_3$ in $\R^2$, with $e_3$ being the canonical versor in $\R^3$ that is orthogonal to the plane. The intensity $b_\eps: \R^2 \to \R$ of the magnetic field is bounded and has a jump or fast transition across a simple curve $\Gamma$ that is closed and $C^4$. If $\Omega$ denotes the compact set of $\R^2$ such that $\partial\Omega =\Gamma$, the simplest example of magnetic field that we consider in this paper is of the form
\begin{align}\label{basic.iwa}
b_\eps(\cdot) = b(\cdot), \ \ \ \ \ b(x) = \begin{cases}
b_+ \ \ \ &\text{if $x \in \Omega$}\\
b_- \ \ \ \ &\text{otherwise}
\end{cases}, \ \ \ \ \ \text{for two values $b_+, b_- > 0$, $b_+ \neq b_-$.}
\end{align}
Our study, however, also includes intensities $b_\eps$ that are not piecewise constant and such that, at any point $x_0 \in \Gamma$, the function $b_\eps$ changes abruptly along the normal direction (c.f. \eqref{def.b.eps}).

\vv

Given $H_\eps$ as in \eqref{hamiltonian.intro}, we consider the spectral problem
\begin{align}\label{spectral.problem}
H_\eps \Psi = \lambda \Psi \ \ \text{in $\R^2$},
\end{align}
for some $\lambda > 0$ and $\Psi \in H^2_{loc}(\R^2; \C)$ such that
\begin{equation}\label{form.domain}
\int_{\R^2} |\Psi|^2 + \int_{\R^2} |(\nabla + i\eps^{-2}a)\Psi|^2 < +\infty.
\end{equation}
For $(\Psi, \lambda)$ solving \eqref{spectral.problem}, the eigenfunction $\Psi$ is an \textit{edge state} whenever $\Psi$ has mass (i.e. $L^2$-norm) that is localized at scale $\eps$ close to $\Gamma$ and that is distributed along this curve. From the point of view of the associated Schr\"odinger's equation, this means that at the energy level $\lambda$, the solution $e^{-i \lambda t}\Psi$ describes a current localized on $\Gamma$ and propagating throughout it (c.f. Subsection \ref{sub.fluxes}). 

\smallskip

The main goal of this paper is to provide a rigorous and detailed description of the edge states for Hamiltonians as in \eqref{hamiltonian.intro}. In particular, we focus on how the variation of the magnetic field $b_\eps$ along $\Gamma$ influences the mass distribution of $\Psi$. 

\smallskip

If the curve $\Gamma$ is a straight line and the magnetic field is as in \eqref{basic.iwa}, problem \eqref{spectral.problem} belongs to the class of \textit{Iwatsuka models} that were first studied in \cite{Iwatsuka}. In this setting, the operator $H_\eps$ may be diagonalised and its spectrum contains an essential part given by the two sets of Landau levels $\{ \frac{b_+}{\eps^2} (n+\frac 1 2)\}_{n\in\N} \cup \{\frac{b_-}{\eps^2}(n+ \frac 1 2)\}_{n\in\N}$ that correspond to the behaviour of the Hamiltonian $H_\eps$ at infinity. The presence of the interface $\Gamma$, across which the magnetic fields jumps, gives rise to an absolutely continuous part of the spectrum that fills all the gaps between the essential part. The (generalised) eigenfunctions associated with the absolutely continuous part are localized at scale $\eps$ close to $\Gamma$. 

\smallskip

The setting of this paper may be thus considered as a generalization of the Iwatsuka models to a regular and compact curve $\Gamma$ and to magnetic fields that are not translation-invariant and may slowly change along $\Gamma$. In the main result of this paper (c.f. Theorem \ref{t.edge}), we identify a subset $\Sigma \subset \R$ such that, if $(\Psi, \lambda)$ solves \eqref{spectral.problem} and $\lambda \notin \Sigma$, then $\Psi$ is an edge state and its $L^2$-norm changes \textit{macroscopically} along $\Gamma$ according to an explicit function. The set $\Sigma$ (c.f. definition \eqref{sigma.set}) contains the bulk part $\{ \frac{b_+}{\eps^2} (n+\frac 1 2)\}_{n\in\N} \cup \{\frac{b_-}{\eps^2}(n+ \frac 1 2)\}_{n\in\N}$ and an additional set that ensures that the function $\Psi$ is localised all along $\Gamma$ and not only on a portion of it.

\smallskip

In fact, we expect that for values $\lambda \in \Sigma \backslash \{ \frac{b_+}{\eps^2} (n+\frac 1 2)\}_{n\in\N} \cup \{\frac{b_-}{\eps^2}(n+ \frac 1 2)\}_{n\in\N}$ the variation of the magnetic field along $\Gamma$ may obstruct the propagation of the corresponding eigenfunction $\Psi$ throughout the full interface $\Gamma$. In analogy with the WKB theory in the semiclassical regime for the Schr\"odinger operator $L_\eps:=-\Delta + \frac{V}{\eps^{2}}$, the points of $\Gamma$ where the propagation yield \textit{turning points} (e.g. \cite[Chapter 15, Subsection 15.4.1]{Hall.book}). As explicitly shown in our main result, the function $\Psi$ does not vanish along $\Gamma$ as long as the solutions to the equivalent of the eikonal equation (c.f \eqref{eikonal}) satisfy a suitable non-transversality condition. We plan to address this scenario in a future paper. 

 \smallskip

The original Iwatsuka model was introduced in \cite{Iwatsuka} for a magnetic Schr\"odinger operator $-(\nabla + i a) \cdot (\nabla + i a)$ in $\R^2$ with a magnetic field $b e_3$ having positive and bounded intensity $b(x_1, x_2)= b(x_2)$, $(x_1, x_2) \in \R^2$, that converges to two distinct constants when $x_2 \to \pm \infty$\footnote{In the same paper, it is also studied the case $b_+=b_-$ when $b$ has a unique and non-degenerate global minimum/maximum.}. This setting provides an example of a magnetic Schr\"odinger operator with purely absolutely continuous spectrum that is generated by the transition of the magnetic field from the two values attained at infinity. 

\smallskip

Since \cite{Iwatsuka}, there is an extensive literature devoted to the study of this class of models and its generalization.  In \cite{Hislop.Soccorsi}, the authors consider a magnetic field of the form $b_\eps= \frac{b_\eps(x_2)}{\eps^2}$, $x_2 \in \R$ that has a fast transition along the line $\Gamma = \{ x_2 = 0\}$: For every $\eps >0$, the magnetic field $b_\eps$ is such that $b_\eps= b_+$ when $x_2 > \eps$ and $b_\eps= b_-$ when $x_2 < -\eps$. In this setting, the authors study the localization at scale $\eps$ around $\Gamma$ for the eigenfunctions corresponding to suitable energies and provide lower bounds for the edge currents carried by such edge states. A similar analysis is performed in \cite{Miranda.Popov} for a magnetic field of intensity $b= b(x_2) \in C^\infty(\R)$ that is monotone and satisfies $\lim_{x \to \pm \infty} b(x) := b_{\pm}$. The study in \cite{Miranda.Popov} relies on a detailed description of the band functions (c.f. the curves $\{ \nu_n: \R \to \R \}_{n\in\N}$ in Subsection \ref{sub.Iwatsuka}, \eqref{def.branches}-\eqref{harmonic.oscillator}) associated to the spectrum of $-(\nabla + i a) \cdot (\nabla + i a)$. Furthermore, it is allows for suitable perturbation of the previous operator by a non-negative electric field $V$.

\smallskip

In \cite{Dombrowski.Hislop, Dombrowski.Raikov}, the localization properties of the edge states are studied in the case of a magnetic field $b$ as in \eqref{basic.iwa} with $\Gamma$ being a line and the two constant limit values satisfy $b_+ > 0,  b_- < 0$. In this case, the change in sign of the magnetic field gives rise to the so-called \textit{snake orbits}. These were first introduced in \cite{ReijniersPeeters} and correspond to the dynamic of a particle for the classical Hamiltonian that lie half on one side of $\Gamma$ and half on the other. In \cite{Dombrowski.Hislop}, an analysis of the snake orbits is brought forward in the case of the magnetic field being anti-symmetric with respect to the line $\Gamma$. In this case, the additional symmetry of the system allows describing the snake orbits and the band functions associated with the magnetic hamiltonian in detail.

\smallskip

We believe that the main novelty of the present paper is the accurate analysis of the edge states, together with an asymptotic expansion for the associated eigenvalues, in the case of magnetic fields $b_\eps$ that have a sharp transition along the normal direction to a general curve $\Gamma$ and that may also change along the direction tangential to it. We refer to \eqref{def.b.eps} in the next section for the detailed assumptions on $b_\eps$. We also stress that our analysis of the edge states is performed at \textit{any} energy level in the gaps of the set $\{ \frac{b_+}{\eps^2} (n+\frac 1 2)\}_{n\in\N} \cup \{\frac{b_-}{\eps^2}(n+ \frac 1 2)\}_{n\in\N}$.

\smallskip

The techniques used in this paper are an extension of the methods developed in \cite{GV}. In the latter, the study of edge states is brought forward in the case of a constant magnetic field $b_\eps = \frac{\bar b}{\eps^2}$, $\bar b \in \R$ and when \eqref{spectral.problem} is solved in a bounded (regular) domain $\Omega$ with Dirichlet boundary conditions. In this case, the presence of the boundary $\partial \Omega$ plays the role of the interface $\Gamma$ in the current paper and gives rise to a discrete part of the spectrum that ``fills'' the gaps between the Landau levels $\{\frac{\bar b}{\eps^2}(n + \frac 1 2) \}_{n\in\N}$. The main result of \cite{GV} shows that, whenever $\lambda$ is between \textit{any} two Landau levels, then the corresponding eigenfunction is an edge state. In contrast with the current paper, its mass is distributed asymptotically uniformly along $\partial\Omega$. {For other results in the literature related to this setting, we refer to \cite{barbaroux.ecc, HelfferFournais, Raymond_book} and to the introduction of \cite{GV} for a more detailed overview of the literature.}

\smallskip

{\bf Structure of the paper and notation.} This paper is organised as follows: In the next section, we introduce the main setting and the main results (Theorems \ref{t.edge}-\ref{t.eigenvalue}). Theorem \ref{t.edge} provides a description of the edge states, while Theorem \ref{t.eigenvalue} gives an asymptotic approximation to the associated eigenvalues. In Subsection \ref{sub.constant} we comment on how the two previous theorems greatly simplify in the case of magnetic fields that are constant along $\Gamma$ (e.g. the one in \eqref{basic.iwa}), while in Subsection \ref{sub.asy} we provide the precise asymptotic approximation for the eigenfunctions close to the interface $\Gamma$ (Proposition \ref{p.asy}). In Section \ref{s.proofs} we prove the main results and carefully comment on the analogies and differences between the current strategy and the one used in our previous paper \cite{GV}. Finally, in the Appendix, we prove and state the auxiliary results that we use throughout the proofs of Section \ref{s.proofs}, including an overview of the main well-known results obtained for the standard Iwatsuka model (Subsection \ref{sub.Iwatsuka}). 

\smallskip

Throughout this paper, we adopt the following notation:

\smallskip

\noindent $\bullet $\, We denote by $\T$ the unitary circle and, for every $\xi, \tilde\xi \in \T$, we write $d(\xi, \tilde\xi)$ for the distance  on $\T$ between the two points.\\
$\bullet$ \,  Given a bounded set $U \subset \R^d$, we denote by $\fint_U$ the averaged (Lebesgue) integral $|U|^{-1} \int_U$, where $|U|$ is the usual (Lebesgue) measure of the set $U$.\\
$\bullet$\,  Given two families $\{ \alpha_\eps\}_{\eps>0}, \{\beta_\eps \}_{\eps >0} \subset \R$ such that $\alpha_\eps \to 0$, we use the notation $\beta_\eps = o( \alpha_\eps)$ if $\frac{\beta_\eps}{\alpha_\eps} \to 0$ when $\eps \to 0$. \\
$\bullet$ \, For $a, b \in \R$ we use the notation $a \wedge b$ for the minimum between $a$ and $b$.

\section{Setting and main results}\label{s.general}
Let $\Gamma$ be a  $C^4$ closed and simple curve in $\R^2$. With no loss of generality, we assume that the curve $\Gamma$ has unitary length. We denote by $f=(f_1(\xi), f_2(\xi))$, $\xi \in \T$ a parametrization of $\Gamma$ according to arc-length and write $(\vT(\xi), \vN(\xi))$ and $\kappa(\xi)$ for the \textit{tangent}, \textit{(outer) normal} and the \textit{curvature} of $\Gamma$ at a point $\xi \in \T$. Since $\Gamma$ is $C^4$, there exists a tubular neighbourhood $U \subset \R^2$ of $\Gamma$ where the change of coordinates
\begin{align}\label{local.coordinates}
U \ni x \mapsto (\xi, s) \in \T \times \R, \ \ \ \ x = f(\xi) - s \vN
\end{align}
is well-defined. 

\vv

Let $b: \T \times \R \to \R$ be any function such that
\begin{itemize}
\item[(A1)] For almost every $s\in \R$, $b$ is twice differentiable in the periodic variable $\xi \in \T$ and $b, \partial_\xi b, \partial_\xi^2 b \in L^\infty(\T \times \R)$;

\item[(A2)] There exist two values $b_+, b_- > 0$, $b_+ \neq b_-$, and $M > 0$ such that for every $\xi \in \T$ the function $b(\xi, s) = b_+$ in $\{s > M\}$ and $b(\xi, s) = b_-$ in $\{s < -M \}$.

\item[(A3)] There exists $m >0$ such that $b(\xi,s) \geq m$  for almost every $(\xi,s) \in \T \times \R$.
\end{itemize}

Given the tubular neighbourhood $U$ and the function $b$ introduced above, for every $\eps >0$ we define the magnetic field $b_\eps$ as 
\begin{align}\label{def.b.eps}
b_\eps(\xi, s) = b(\xi, \frac{s}{\eps}) \ \ \ \ \text{in $U$}
\end{align}
and continuously extend it to be $b_+$ or $b_-$ in the two connected components of $\R^2 \backslash U$. With no loss of generality, we assume throughout the paper that $b_- < b_+$.

\vv

For every $\eps> 0$ we thus consider the Hamiltonian $H_\eps$ in \eqref{hamiltonian.intro} with $b_\eps$ as above. The spectrum $\sigma( H_\eps)$ has an essential part given by $\sigma_{\text{ess}}= \{\frac{b_-}{\eps^2}(n + \frac 1 2) \}_{n\in \N}$ and, away from this set, the spectrum is discrete (see, for instance, \cite{RozTas} and \cite[Theorem 2.1]{AvronHerbstSimon}). 

\vv

If an edge state for $H_\eps$ is localized at scale $\eps$ around $\Gamma$, after a suitable blow-up around a point of $\Gamma$, the new ``magnified'' problem \eqref{spectral.problem} is expected to resemble the Iwatsuka model of Subsection \ref{sub.Iwatsuka} with the choice $b= b(\xi, \cdot)$, $ s\in \R$. This motivates the introduction of the following notation that is needed to state the main theorems.

For every $\xi \in \T$ fixed, let $\Hiwa$ be the Hamiltonian of Section \ref{sub.Iwatsuka} with magnetic field $b(\xi, \cdot)$. Let $\{ \mathcal{O}(\xi, k)\}_{k\in \R}$ be the family of associated one-dimensional operators
\begin{align}\label{harmonic.oscillator.general}
\mathcal{O}(\xi, k):= - \partial_s^2 + (\int_0^s b(\xi, t) \, \d t - k)^2 \ \ \ \text{in $\R$.}
\end{align}
For every $k \in \R$ fixed, the assumptions on $b$ yield that $\mathcal{O}(\xi, k)$ has a discrete spectrum $\{ \nu_n(\xi, k) \}_{n\in \N}$ (c.f. Subsection \ref{sub.Iwatsuka} and Lemma \ref{l.branches}). For every $n \in \N$, we thus define the functions:
\begin{align}\label{def.nu}
\nu_n : \T \times \R \to \R, \ \ (\xi, k) \mapsto \nu_n(\xi, k) \ \text{the $n^{th}$ eigenvalue of $\mathcal{O}(\xi, k)$ with magnetic field $b(\xi; \cdot)$.}
\end{align}
By Lemma \ref{l.branches} applied to $b= b(\xi, \cdot)$, the previous curves are differentiable in the variable $k$. Equipped with this notation, we define the sets
\begin{equation}
\begin{aligned}\label{sing.set}
\sigma_{\text{bulk}}&:= \{  b_- (n+\frac 12 ) \}_{n\in\N} \cup \{ b_+(n + \frac 1 2) \}_{n\in \N},\\
\sigma_{\text{sing}}&:= \{ \lambda \in \R \, \, \colon \, \, \text{there exists $n\in \N$, $(\xi, k) \in \T\times \R$ such that } \nu_n(\xi , k) = \lambda, \, \partial_k \nu_n(\xi, k) = 0 \},
\end{aligned}
\end{equation}
and 
\begin{align}\label{sigma.set}
\Sigma:= \sigma_{\text{bulk}} \cup \sigma_{sing}.
\end{align}

\vv

The definition of $\Sigma$ yields that if $\lambda \notin \Sigma$, then there exists $N \in \N$ (possibly zero) such that for every $\xi \in \N$ there exist exactly $N \in \N$ values $\{ k_j(\xi) \}_{j =1}^N \subset \R$ such that for each $j =1 ,\cdots, N$, there exists a unique $n_j \in \N$ such that
\begin{align}\label{eikonal}
\nu_{n_j}(\xi ; k_j(\xi))= \lambda
\end{align}
Furthermore, the curves $k_j : \T \to \R$, $\xi \mapsto k_j(\xi)$ are well-defined and $C^2$ for every $j=1, \cdots, N$. The previous claims are an easy consequence of the regularity of the surfaces $\nu_l: \T \times \R \to \R$ (see Lemma \ref{l.branches}), standard topological arguments and the Implicit Function Theorem. We stress, in fact, that the definition \eqref{sigma.set} allows for the Implicit function theorem to be applied and infer the existence of the curves $\{ k_j\}_{j=1}^N$. We postpone the detailed proof to the Appendix (c.f. Lemma \ref{l.implicit.function}).

\medskip

{\begin{rem}
For a general magnetic field as in \eqref{def.b.eps}, having limit values $b_+, b_- > 0$, and a given $\lambda \notin \Sigma$, the number $N$ of solutions to \eqref{eikonal} admits the lower bound $N \geq n_1 - n_2$, where $n_1, n_2 \in \N$ are such that 
$$
b_-( n_1 + \frac{1}{2}) < \lambda < b_-( n_1 + \frac{3}{2}), \ \ \ \ b_+( n_2 + \frac{1}{2}) < \lambda < b_+( n_2 + \frac{3}{2}).
$$
We stress that $n_1, n_2 \in \N$ do exists since $\lambda \notin \sigma_{\text{bulk}}$. The value $n_1- n_2$ may be characterised using the so-called Chern number { \cite[Chapter 3]{Bernevig.book}}.  We also remark that, if the magnetic field $b_\eps$ is monotone in the variable $s$, then also the functions $\nu_l(\cdot, \cdot)$ are monotone in $k \in \R$ (c.f. Lemma \ref{l.branches}) and $N= n_1-n_2$.
\end{rem} }

\medskip

The next result states that, for energies away from the set $\Sigma$, the corresponding eigenfunction $\Psi_\eps$ is an edge state. More precisely, $\Psi_\eps$ is localised around $\Gamma$ (i.e. \eqref{localization}) and we give a precise description of how its $L^2$-norm is asymptotically distributed along $\Gamma$ (i.e. \eqref{propagation}).

\begin{thm}[Asymptotic behaviour of the edge states]\label{t.edge}
Let $b_\eps$ and $\Sigma$ be as above. Let $\{\Psi_\eps , \lambda_\eps)_{\eps > 0}$ be a family of solutions to \eqref{spectral.problem}-\eqref{form.domain} such that $\| \Psi_\eps \|_{L^2(\R^2)} =1$.  We assume that $\eps^2\lambda_\eps \to \lambda$ with $\lambda \notin \Sigma$. Then $\Psi_\eps$ is an \textit{edge state} and:

\begin{itemize}
\item[(a)] if $d_{\Gamma}$ denotes the distance function from the boundary $\Gamma$, then for every $n\in \N$ there exists a constant $C=C(n, \lambda, \Gamma)$ such that
\begin{align}\label{localization}
 \| (d_{\Gamma} \wedge 1)^n \Psi_\eps \|_{L^2(\R^2 ; \C)} +  \eps^2  \| (d_{\Gamma} \wedge 1)^n (H_\eps \Psi_\eps )\|_{L^2(\R^2; \C)} \leq C \eps^n.
 \end{align}

\item[(b)] Let $\{r_\eps\}_{\eps>0}$ be such that $\eps^{-1} r_\eps \to +\infty$ and $\eps^{-\frac 1 2}r_\eps \to 0$. Let $\{k_{l}\}_{l=1}^N \subset C^1(\T)$ be the curves that solve \eqref{eikonal} for the limit value $\lambda$. Then, there exists a sequence $\{\eps_j\}_{j\in \N}$ and $ \{ A_\ell \}_{\ell=1}^N \subset \C$ with $\sum_{\ell=1}^N |A_\ell|=1$, such that for every $x_0=(\xi, 0) \in \Gamma$ we have
\begin{align}\label{propagation}
\lim_{j\to \infty} \bigl( (2r_{\eps_j})^{-1}\int_{ |x - x_0| < r_{\eps_j}} |\Psi_{\eps_j}(x)|^2 \, \d x \bigr)^{\frac 1 2} = \sum_{\ell =1}^N |A_\ell| \frac{|\partial_k \nu_{n_\ell}(\xi, k_{\ell}(\xi))|}{\bigl(\fint_{\T}|\partial_k \nu_{n_\ell}(y, k_\ell(y))|^2 \, \d y\bigr)^{\frac 1 2}}.
\end{align}
\end{itemize}
\end{thm}

\begin{rem}\label{rem.turning.points}
In contrast with the analogous result for a constant magnetic field and Dirichlet boundary conditions \cite[Theorem 2.5 and limit (2.24)]{GV}, the change of the magnetic field $b$ along $\Gamma$ yields that the amplitude of the $L^2$-norm of $\Psi_\eps$ \textit{changes macroscopically} along $\Gamma$. The amplitude of $\Psi_\eps$, in particular, depends on the values of the derivatives $\partial_k \nu_{n_i}(\xi, k_j(\xi))$.  Thanks to the assumption  $\lambda \notin \Sigma$,  these are bounded both from above and away from zero. This implies that all the ratios in the sum on the right-hand side of \eqref{propagation} are bounded and never vanish along $\T$.
\end{rem} 

\begin{rem}\label{rem.resonances}
We stress that if the number of curves solving \eqref{eikonal} is $N=1$, then the sequence $\{r_\eps\}_{\eps >0}$ in Theorem \ref{t.edge}, part $(b)$ may be chosen as $r_\eps = \eps$. As further discussed below in Proposition \ref{p.asy}, Theorem \ref{t.edge}, (b) follows from a detailed asymptotic formula for the eigenfunctions $\Psi_\eps$ close to the interface $\Gamma$. This formula allows to approximate $\Psi_\eps$ by a superposition of functions that oscillate in the variable $\xi$ as the wave functions $e^{\frac i \eps \int_0^\xi k_\ell(x) \, \d x}$, $\ell =1, \cdots, N$.  Therefore, when $N> 1$ the previous waves may interact along lenghtscales $\xi \sim \eps$, but do become decoupled along any mesoscopic scale $r_\eps >> \eps$. In other words, for $\ell, j =1 ,\cdots, N$ such that $\ell \neq j$ it holds
\begin{align}\label{resonances}
|\int_{|\xi | < r_\eps} e^{\frac i \eps \int_0^\xi (k_\ell(x)- k_j(x)) \, \d x } \, \d\xi | \leq \frac{\eps}{\delta r_\eps},
\end{align}
whenever $|k_i(\xi)- k_j(\xi)| > \delta$ for every $\xi\in \T$. Since this last inequality is satisfied by the curves $\{ k_{\ell}\}_{\ell =1}^N$ thanks to Lemma \ref{l.implicit.function}, the right-hand side above vanishes in the limit whenever $\frac{r_\eps}{\eps} \to +\infty$. This technical issue is the same that arises in \cite{GV} and that distinguishes \cite[Theorem 2.4]{GV} from \cite[Theorem 2.5]{GV} (see also \cite[Formulas (2.20)-(2.21)]{GV} for a further comment on this).
\end{rem}

\smallskip

The next main result provides an asymptotic expansion for the eigenvalues of $H_\eps$ in \eqref{spectral.problem} that correspond to edge states. This result should be compared with \cite[Corollary 2.6]{GV} that is the analogue in the case of constant magnetic fields and Dirichlet boundary conditions. We stress that the high generality of the magnetic fields $b_\eps$ considered in this paper yields that the asymptotic expansion for the eigenvalues $\lambda_\eps \in \sigma(H_\eps)$ is given in terms of functions $\Lambda_i$, $i=1, \cdots, N$ that are implicitly defined. In fact, in the case of magnetic fields that do not change along $\Gamma$, the next theorem turns into an easier asymptotic approximation for the eigenvalues (see Corollary \ref{t.eigenvalue.simple} in the next subsection).

\begin{thm}[Asymptotic expansion for the eigenvalues]\label{t.eigenvalue}
Let $I \subset \R$ be an open and bounded interval such that dist$(I, \Sigma) > 0$. Then:
\begin{itemize}

\item[(a)] There exists $N \in \N$ and smooth curves $K_j= K_j(\lambda, \xi)$, $j=1, \cdots, N$ such that, for every $\lambda \in I$, the function $K_j(\lambda, \cdot)$ solves equation \eqref{eikonal}. Moreover, for every $j=1, \cdots, N$  the map
\begin{align}
\Lambda_j: I \to \R, \ \ \ \ \lambda \mapsto \Lambda_j(\lambda):= \int_{\T} K_j(\lambda, \xi) \d \xi.
\end{align}
is differentiable and invertible. 

\item[(b)] There exists $\eps_0$ such that for all $\eps \leq \eps_0$, every $\lambda_\eps \in \sigma( H_\eps)$ such that $\eps^2 \lambda_\eps \in I$ satisfies
$$
\eps^2\lambda_\eps = \Lambda_j^{-1}(q_\eps) + \eps  \int_\T \frac{B_{n_j}(\xi, K_j(\Lambda_j^{-1}(q_\eps), \xi))}{D_{n_j}(\xi, K_j(\Lambda_j^{-1}(q_\eps), \xi))} + o(\eps)
$$
for some $j=1, \cdots, N$, $q_\eps \in 2\pi\eps \Z$ and where, for every $n\in \N$, the functions $D_n, B_n: \T \times \R \to \R$ are defined as
\begin{equation}
\begin{aligned}\label{C.ell.hard}
B_n (\xi, k) &= 2 \kappa(\xi) \int \biggl (k \int_0^t (\tilde t-\mu) b(\xi, \tilde t) \, \d \tilde t + (\int_0^t  \tilde t\, b(\xi, \tilde t) \, \d \tilde t)(\int_0^t b(\xi, \tilde t) \, \d \tilde t)\biggr) H_n(\xi, k, t)^2 \, \d t\\
D_n(\xi, k) &= \frac{1}{\partial_k \nu_n(\xi, k)}\biggl( \int_{\T} \frac{1}{\partial_k \nu_{n}(\tilde \xi, k)} \, \d \tilde \xi \biggr)^{-1}.
\end{aligned}
\end{equation}
We recall that here $\kappa$ denotes the curvature of $\Gamma$ and, for every $n\in \N$ and $k\in \R$,  the functions {$H_n(k,\xi, \cdot)$ are the eigenfunctions associated to the operator $\O(\xi, k)$ in \eqref{harmonic.oscillator.general} and associated to the eigenvalue $\nu_n(k,\xi)$.}
\end{itemize}
\end{thm}

\subsection{Edge currents}\label{sub.fluxes}
In this subsection we show that the edge states described in Theorem \ref{t.edge} do carry a current. Let $(\Psi_\eps, \lambda_\eps)_{\eps> 0}$ solve \eqref{spectral.problem}-\eqref{form.domain}: In line with \cite{Hislop.Soccorsi, Hornberger.Smilansky}, we define the flux 
\begin{align}\label{def.flux}
j_\eps(x):= 2 \Imm\bigl(\overline{\Psi}_\eps (\nabla + i \frac{a_\eps}{\eps^2}) \Psi_\eps\bigr),
\end{align}
where $\overline{\Psi}_{\eps}$ denotes the complex conjugate of $\Psi_\eps$ and $\Imm(z)$, $z \in \C$, is the imaginary part of $z$. 

\smallskip

Let $(\Psi_\eps, \lambda_\eps)$ be edge states as in Theorem \ref{t.edge} and let $j_\eps$ be the flux associated to this choice of $\Psi_\eps$. Appealing to Theorem \ref{t.edge}, $(a)$, it is an easy consequence of \eqref{def.flux} and Cauchy-Schwarz inequality that for every $\gamma \in (0,1)$ and every $n\in \N$ there exists a constant $C= C(n, \gamma, \lambda, \Gamma)$ such that
\begin{align}\label{loc.A}
\int_{\mathop{dist}(x, \Gamma) > \eps^{\frac 2 3}} |j_\eps(x)| &\leq C \eps^n,
\end{align}
namely the fluxes concentrate around $\Gamma$. The next result states that $j_\eps$ does not vanish on $\Gamma$ and  that it is asymptotically concentrated along this curve in the tangential direction $\vT$:
\begin{cor}\label{cor.fluxes}
Let $(\Psi_\eps, \lambda_\eps)_{\eps>0}$ satisfy the hypotheses of Theorem \ref{t.edge}. Then, up to a subsequence, there exist $\{ A_\ell \}_{\ell =1}^N \subset \C$ with $\sum_{\ell =1}^N |A_\ell|^2 = 1$ such that 
\begin{align*}
\eps j_\eps \to \beta \delta_{\Gamma}  \vT \ \ \ \text{ in $\mathcal{D}'(\R^2; \R^2)$},
\end{align*}
where $\delta_\Gamma$ denotes the Dirac measure concentrated on $\Gamma$ and the function $\beta : \T \to \R$ is defined as 
$$
\beta(\xi) := \sum_{\ell =1}^N |A_\ell |^2 \frac{|\partial_k \nu_{n_\ell}(\xi, k_\ell(\xi))|^2}{\fint_\T |\partial_k \nu_{n_\ell}(y, k_\ell(y))|^2 \, \d y} \partial_k \nu_{n_\ell}(\xi, k_\ell(\xi)).
$$
Here, the space $\mathcal{D}'(\R^2; \R^2)$ is the space of $\R^2$- valued distributions on $\R^2$.
\end{cor} 
We recall that, thanks to the definition \eqref{sigma.set} of the set $\Sigma$, the function $\beta$ never vanishes on $\T$. 

\subsection{Magnetic fields that are constant along $\Gamma$}\label{sub.constant}
In the special case $b(\xi,s)=b(s)$, that includes example \eqref{basic.iwa} in the Introduction, the curves $\{k_i(\xi) \}_{i=1}^N$ solving \eqref{eikonal} are constant (i.e. $k_i(\xi) \equiv k_i \in \R$ for every $\xi \in \T$) and the functions $\nu_n$, $n\in \N$ do not depend on the variable $\xi \in \T$. In this case, inequality \eqref{propagation} of Theorem \ref{t.edge}, part (b), turns into:
\begin{align*}
\lim_{j\to \infty} (2 r_\eps)^{-1} \int_{ |x - x_0| < r_\eps} |\Psi_\eps(x)|^2 \, \d x = 1,
\end{align*}
which means that in this case the mass of the eigenfunctions $\Psi_\eps$ is asymptotically uniformly distributed along $\Gamma$.

\vv

Moreover, the statement of Theorem \ref{t.edge} may be simplified into the following:
\begin{cor}\label{t.eigenvalue.simple} Let $(\Psi_\eps, \lambda_\eps)$ be as in Theorem \ref{t.eigenvalue}. Then, for some $l \in \N$ and $q_\eps \in 2\pi\eps\Z$
\begin{align}\label{characterization.spectrum}
\eps^2\lambda_\eps=   \nu_{l}(q_\eps) + 4\pi\eps \bar B_{l}(q_\eps) + o(\eps),
\end{align}
where the function $B_l: \R \to \R$ is defined as 
\begin{align}\label{C.ell}
\bar B_l (k) = \int \biggl (k \int_0^t (\tilde t-\mu) b(\tilde t) \, \d \tilde t + (\int_0^t  \tilde t\, b(\tilde t) \, \d \tilde t)(\int_0^t b(\tilde t) \, \d \tilde t)\biggr) H_l(k, t)^2 \, \d t , \ \ \ k \in \R.
\end{align}
\end{cor}

\vv

We also stress that, in this case, the function $\beta$ in Corollary \ref{cor.fluxes} is constant and equals
$$
\beta \equiv \sum_{\ell =1}^N |A_\ell|^2 \partial_k \nu_{n_\ell}( k_\ell).
$$

We finally remark that if $b$ is monotone (as in \eqref{basic.iwa}), then Lemma \ref{l.branches} yields that the branches $\nu_l(\xi, k) \equiv \nu_l(k)$ are monotone as well. This implies that the set $\sigma_{\text{sing}}$ is empty and that $\Sigma$ coincides only with the bulk part $\sigma_{\text{bulk}}$.

\subsection{Precise asymptotic expansion}\label{sub.asy}
in analogy with the main results in \cite{GV}, Theorem \ref{t.edge}, part (b) and Theorem \ref{t.eigenvalue} are an easy consequence of a precise asymptotic information on the behaviour of the eigenfunctions $\Psi_\eps$ close to $\Gamma$: Let $(\Psi_\eps, \lambda_\eps)$ satisfy the assumptions of Theorem \ref{t.edge}. Then, for every $\eps$ small enough also $\eps^2\lambda_\eps \notin \Sigma$ and there are also exactly $N$ smooth curves $\{ k_{i,\eps} \}_{i=1}^N$ satisfying \eqref{eikonal} with $\lambda$ replaced by $\eps^2 \lambda_\eps$. 
\vv

Before stating the next proposition, we need the following notation: For every curve $k_j=k_j(\xi)$, $j=1, \cdots, N$ solving \eqref{eikonal}, we define the function
\begin{align}
W_{j,\eps}(\xi, k_n(\xi), s):= e^{\frac i \eps \int_0^\xi k_{j,\eps}(y) \, \d y} H_{n_j}(\xi, k_n(\xi), s), \ \ \ \ \text{$(\xi, s) \in \T \times \R$.}
\end{align}
We recall that for every $n \in \N$ and $(\xi, k) \in \T$, $H_n(\xi, k(\xi), \cdot)$ is the eigenfunction of $\mathcal{O}(\xi, k)$ associated to the eigenvalue $\nu_n(k,\xi)$. We refer to Lemma \ref{l.branches} in Subsection \ref{sub.Iwatsuka} for the properties of these functions.

\vv

\begin{prop}\label{p.asy}
Let $(\Psi_\eps, \lambda_\eps)_{\eps>0}$ and the sequence $\{r_\eps\}_{\eps>0} \subset \R_+$ be as in Theorem \ref{t.edge}. Then there exists a global gauge $\theta_\eps$ such that the the function $\tilde \Psi_\eps = e^{i\theta_\eps}\Psi_\eps$ satisfies the following property: For every subsequence $\{\eps_j\}_{j \in \N}$, there exist $A_1, \cdots A_N \in \C$ with $\sum_{j=1}^N |A_j|^2 =1$ such that 
\begin{align*}
\tilde \Psi_{\eps_j} \sim  \Psi_{\text{flat}, \eps_j}:= \eps_j^{-\frac 1 2}\sum_{l=1}^N A_l  \frac{|\partial_k\nu_{n_l}(\xi, k_l(\xi))|} {\bigl(\fint_\T |\nu_{n_l}(y, k_l(y))|^2 \, \d y\bigr)^{\frac 1 2}} e^{i \int_0^\xi \frac{B_l(y, k_{l}(y))}{\partial_k\nu_{n_l}(y, k_l(y))}  \, \d y}W_{j,\eps}(\xi, k_{l,\eps}(\xi), \frac s \eps),
\end{align*}
where $B_l$ is as in \eqref{C.ell.hard}, and in the sense that
\begin{align}\label{closeness.to.flat.theorem}
\lim_{j \uparrow +\infty}\sup_{\xi^* \in \T}\bigl(\fint_{d(\xi, \xi^*)<  r_{\eps_j}} \int |\tilde\Psi_{\eps_j} - \Psi_{\text{flat}, \eps_j}|^2 \d\xi \, \d s \bigr)^{\frac 1 2}= 0.
\end{align}
\end{prop}

\section{Proofs}\label{s.proofs}

In the remaining part of the paper, for any $a, b \in \R$ we use the notation $a \lesssim b$ and $a \gtrsim b$ if  there exists a constant $C$ depending on $\Gamma$, $|\lambda|$, $\mathop{dist}(\lambda, \Sigma)$ and $m, M$ in assumptions (A1)-(A3) for $b$ such that $a \leq C b$ and $a \geq C b$, respectively.

\subsection{Proofs of Theorems \ref{t.edge}, \ref{t.eigenvalue} and Corollary \ref{t.eigenvalue.simple}}

In this section, we show how to adapt the proofs of \cite{GV} to the current setting. Both the strategy and most of the auxiliary results contained in \cite{GV} may be easily adapted also to the case of magnetic fields as in \eqref{def.b.eps}. Below, we thus provide an overview of the strategy that we use to prove Theorem \ref{t.edge}, \ref{t.eigenvalue} and give the details for the parts that conceptually differ from the previous paper.

\smallskip

\begin{proof}[Proof of Theorem \ref{t.edge}]
The proof for Part $(a)$ is similar to the one for \cite[Proposition 2.3]{GV} and the main difference in this setting is that the domain is the whole space  $\R^2$. If $\Omega$ is the bounded set that has boundary $\Gamma$, we argue the inequality of part $(a)$ separately in $\Omega$ and $\R^2 \backslash \Omega$. 

\smallskip

The argument for the set $\Omega$ is analogous to the one for \cite[Proposition 2.3]{GV} and the only difference is that the cut-off function $\phi= \phi_\eps$ used for \cite[formula (3.1)]{GV} solves \cite[boundary value problem (2.12)]{GV} in $\Omega_\eps = \{ x \in \Omega \, \colon \, \dist(x; \Gamma) > M\eps \}$, where $M$ is as in (A2). Since $\Gamma$ is $C^4$, this set is regular enough for $\eps$ sufficiently small and the magnetic field $b_\eps \equiv b_+$ in $\Omega_\eps$ (c.f. \eqref{def.b.eps}). We also stress that the boundedness of $b_\eps$ (c.f. (A1)) is enough to infer, by standard elliptic regularity, that $\Psi_\eps \in H^2_{loc}(\R^2)$. 

\smallskip

We now turn to the set $\R^2 \backslash \Omega$: By a standard partition argument, it suffices to prove that for every $\eta \in C^\infty_0(\R^2 \backslash \Omega)$
\begin{align*}
 \int \eta^2 |d_{\Gamma} \wedge 1|^{2n} |\Psi_\eps|^2 +  \eps^2  \int \eta |d_{\Gamma} \wedge 1|^{2n} |H_\eps \Psi_\eps|^2  \lesssim C(n) \eps^n \int |\Psi_\eps|^2 \1_{\mathop{supp}(\eta)}.
\end{align*}
The argument for this inequality is similar to the one above, provided that we use the cut-off function $\phi = \eta \phi_\eps$, with $\eta$ as above and $\phi_\eps$ the solution of \cite[boundary value problem (2.12)]{GV} in the exterior domain $\{ x \notin \Omega \, \colon \, \dist(x ; \Gamma) > M\eps \}$.

\smallskip

We now turn to the proof of Theorem \ref{t.edge}, (b): In the case $N=1$, this follows easily from (a) and from the asymptotic expansion of Proposition \ref{p.asy} together with the properties of the functions $W_{j,\eps}$ (c.f. also Lemma \ref{l.branches}). For $N > 1$ the proof is similar and relies on the choice of the mesoscale $\{ r_\eps \}_{\eps>0}$ that, thanks to \eqref{resonances} and Lemma \ref{l.implicit.function} implies that
\begin{align*}
\lim_{\eps \to 0} \bigl |\int_{|\xi | < r_\eps} e^{\frac i \eps \int_0^\xi (k_i(x)- k_j(x)) \, \d x } \, \d\xi \bigr | = 0.
\end{align*} 
\end{proof}

\begin{proof}[Proof of Theorem \ref{t.eigenvalue}]
We begin by showing part (a): The existence of the value $N$ and the curves $\{ k_i(\lambda, \xi) \}_{i=1}^N \in C^1(I \times \T)$ follows the same argument of Lemma \ref{l.implicit.function} if we apply the Implicit Function Theorem to the functions $F_{n_i}: I \times \T \times \R \to \R$, $F_{n_i}(\lambda, \xi, k):= \nu_{n_i}(\xi, k) - \lambda$. Also in this case, the fact that $I$ is an interval that satisfies the assumption $\mathop{dist}(I, \Sigma) > 0$, allows to define the curves $\{k_i\}_{i=1}^N$ globally over the set $I \times \T$. We remark that, since $\lambda \notin \Sigma$, the partial derivatives 
$$
\partial_\lambda k_i(\lambda, \xi) = \frac{\partial_\lambda F_i(\lambda, \xi, k_i(\lambda, \xi))}{\partial_k F_i(\lambda, \xi, k_i(\lambda, \xi))}= - \frac{1}{\partial_k\nu_{n_i}(\xi, k_i(\lambda, \xi))} \neq 0 \ \ \text{ for all $\xi \in \T$ and $\lambda \in I$.}
$$
and, by continuity, they do have a sign.

\smallskip

By the previous argument it follows that, for each $i=1, \cdots, N$, the function $\Lambda_i$ defined as in the statement of Theorem \ref{t.eigenvalue} $(a)$, is well-defined and its derivative
\begin{align}\label{derivative.inverse}
\Lambda_i'(\lambda) = \int_\T \partial_\lambda k_i(\lambda, \xi) \, \d\xi = - \int_\T \frac{1}{\partial_k\nu_{n_i}(\xi, k_i(\lambda, \xi))} \, \d \xi
\end{align}
has a sign. This implies that $\Lambda_i$ is monotone in $I$ and its inverse is well-defined and continuous. This establishes the statement of part $(a)$.

\medskip

We now turn to part $(b)$: Using the expansion of Proposition \ref{p.asy}, we may argue as for \cite[Proof of Corollary 2.6, (4.19)-(4.20)]{GV} and infer that the periodicity in the variable $\xi$ of the functions $\Psi_\eps$ and $\nu_n$ yields
\begin{align*}
\int_\T k_{i}( \eps^2\lambda_\eps, \xi) \, \d\xi + \eps \int_\T \frac{B_i(\xi, k_{i}( \eps^2\lambda_\eps, \xi))}{\partial_k \nu_{n_i}(\xi, k_i(\eps^2\lambda_\eps, \xi))} \d\xi = q_\eps + o(\eps), \ \ \ \ \text{for some $q_\eps \in 2\pi \eps\Z$.}
\end{align*}
Above, we used the fact that the functions $k_{i,\eps}$ of Proposition \ref{p.asy} may be rewritten, with the notation of part $(a)$, as $k_{i,\eps}(\xi)= k_i(\eps^2\lambda_\eps, \xi)$. Using the definition of $\Lambda_i$, we rewrite the previous identity as
\begin{align*}
\Lambda_i(\eps^2 \lambda_\eps) = q_\eps - \eps \int_\T \frac{B_i(\xi, k_{i}( \eps^2\lambda_\eps, \xi))}{\partial_k \nu_{n_i}(\xi, k_i(\eps^2\lambda_\eps, \xi))} \d\xi + o(\eps), \ \ \ \ \text{for some $q_\eps \in 2\pi \eps \Z$.}
\end{align*}
By part $(a)$ and \eqref{derivative.inverse} this also implies that
\begin{align*}
\eps^2 \lambda_\eps = \Lambda_i^{-1}(q_\eps) + \eps \bigl(\int_\T \frac{1}{\partial_k\nu_{n_i}(\xi, k_i(\Lambda_i^{-1}(q_\eps), \xi))} \, \d \xi \bigr)^{-1} \int_\T \frac{B_i(\xi, k_{i}( \eps^2\lambda_\eps, \xi))}{\partial_k \nu_{n_i}(\xi, k_i(\eps^2\lambda_\eps, \xi))} \d\xi + o(\eps).
\end{align*}
By the regularity of the functions $k_i$, $\nu_{n_i}$ and since $\lambda \in I$, the second term on the right-hand side is of size $\eps$. Hence, $\eps^2\lambda_\eps = \Lambda_i^{-1}(q_\eps) + O(\eps)$. Inserting this into the term 
$$
\bigl(\int_\T \frac{1}{\partial_k\nu_{n_i}(\xi, k_i(\Lambda_i^{-1}(q_\eps), \xi))} \, \d \xi \bigr)^{-1} \int_\T \frac{B_i(\xi, k_{i}( \eps^2\lambda_\eps, \xi))}{\partial_k \nu_{n_i}(\xi, k_i(\eps^2\lambda_\eps, \xi))} \d\xi
$$
and using the regularity of all the functions involved in the above formula, we infer that
\begin{align*}
\eps^2 \lambda_\eps = \Lambda_i^{-1}(q_\eps) + \eps \bigl(\int_\T \frac{1}{\partial_k\nu_{n_i}(\xi, k_i(\Lambda_i^{-1}(q_\eps), \xi))} \, \d \xi \bigr)^{-1} \int_\T \frac{B_i(\xi, k_{i}( \Lambda_i^{-1}(q_\eps), \xi))}{\partial_k \nu_{n_i}(\xi, k_i(\Lambda_i^{-1}(q_\eps), \xi))} \d\xi + o(\eps),
\end{align*}
i.e. the desired formula.
\end{proof}

\begin{proof}[Proof of Corollary \ref{t.eigenvalue.simple}]
This statement is an immediate consequence of Theorem \ref{t.eigenvalue}: Since the magnetic field $b_\eps$ does not depend on the angular variable $\xi\in \T$, in this setting the curves $\{ k_i \}_{i=1}^N$ of Theorem \ref{t.eigenvalue}, part $(a)$ do not depend on $\xi$. The map $\Lambda_i$ defined there thus turns into
\begin{align*}
\Lambda_i(\lambda) = \int_\T k_i(\lambda, \xi) \, \d\xi = k_i(\lambda), \ \ \ \ \nu_{n_i}(k_i(\lambda))=\lambda,
\end{align*}
which implies that $\Lambda_i= \nu_{n_i}^{-1}$ in $I$. Inserting this into the asymptotic expansion of Theorem \ref{t.eigenvalue}, part $(b)$ and using that, in this case
\begin{align*}
\bigl(\int_\T \frac{1}{\partial_k\nu_{n_i}(\xi, k_i(\Lambda_i^{-1}(q_\eps), \xi))} \, \d \xi \bigr)^{-1} &\int_\T \frac{B_i(\xi, k_{i}( \Lambda_i^{-1}(q_\eps), \xi))}{\partial_k \nu_{n_i}(\xi, k_i(\Lambda_i^{-1}(q_\eps), \xi))} \d\xi\stackrel{\eqref{C.ell.hard}}{=} \int_\T {B_i(k_{i}( \Lambda_i^{-1}(q_\eps)), \xi)} \d\xi\\
&\stackrel{\eqref{C.ell}}{=} \bar B_i(k_{i}( \Lambda_i^{-1}(q_\eps))) \int_\T \kappa(\xi) \d\xi = 2\pi \bar B_i(k_{i}( \Lambda_i^{-1}(q_\eps))),
\end{align*}
we establish  Corollary \ref{t.eigenvalue.simple}.
\end{proof}

\subsection{Proof of Corollary \ref{cor.fluxes}}
The proof of Corollary \ref{cor.fluxes} relies on Proposition \ref{p.asy} and on the Theorem \ref{t.edge}, $(a)$. 

\begin{proof}[Proof of Corollary \ref{cor.fluxes}]
We begin by arguing the statement in the case $N=1$. The general case $N \in \N$ is only more technical and requires a modification similar to the one that was implemented in the proof of Theorem \ref{t.edge}.

\smallskip

We begin by claiming that
\begin{align}\label{bound.A}
\limsup_{\eps \downarrow 0} \int | \eps j_\eps(x)| &\lesssim 1.
\end{align}
Using \eqref{loc.A} with $\gamma = \frac 13$, it suffices to prove that
 \begin{align}\label{flux.1}
\limsup_{\eps \to 0} \int_{\mathop{dist}(x, \Gamma) < \eps^{\frac 23}} |\eps j_\eps| \lesssim 1
 \end{align}
so that we only work in a small neighbourhood of $\Gamma$ that is contained in the set $U$ where the curvilinear coordinates \eqref{local.coordinates} are well-defined (c.f. \eqref{local.coordinates}). We thus appeal to Proposition \ref{p.asy} and decompose the eigenfunction $\Psi_\eps$ as
\begin{align}\label{decompose.psi}
\Psi_\eps := \Psif + R_\eps,
\end{align}
with $\Psif$ as in the statement of Proposition \ref{p.asy}. We now argue that the error term $R_\eps$ satisfies
\begin{align}\label{error.norm}
\| R_\eps\|_{L^2(U)} + \eps \|(\nabla + i \eps^{-2}a_\eps)R_\eps\|_{L^2(\tilde U)} = o(1),
\end{align}
where  $\tilde U: = \{ x \in U \, \colon \, \mathop{dist}(x, \partial U) < \delta\}$ for some small $\delta >0$ fixed such that $\tilde U \neq \emptyset$.  The first identity for $R_\eps$ is an immediate consequence of  \eqref{closeness.to.flat.theorem} in Proposition \ref{p.asy}, after we sum over a suitable partition of the curve $\Gamma$ made of intervals of size comparable to $r_\eps$. For the second identity above we argue as follows: Writing \eqref{spectral.problem} in the rescaled local coordinate $(\mu, \theta)$ defined in \eqref{micro.coordinates}, Lemma \ref{l.local.hamiltonian} implies that
\begin{align}
(H_0 + \eps H_{1,\eps} + \eps^{2}H_{2,\eps}) \Psi_\eps = \eps^2 \lambda_\eps \Psi_\eps \ \ \ \ \text{in $U$.}
\end{align}
Using the definition of $\Psif$, the properties of $\nu_1$, $k_1$ and $H_1(\cdot, \cdot)$ (c.f. Lemma \ref{l.branches} and \ref{l.implicit.function}) imply that, in the rescaled coordinates $(\mu, \theta)$ we have also
\begin{align}
(H_0 + \eps H_{1,\eps} + \eps^{2}H_{2,\eps}) \Psif = \eps^2\lambda_\eps \Psif + (\eps H_{1,\eps} + \eps^2 H_{2,\eps}) \Psif + \eps f_\eps,
\end{align}
with $\| f_\eps \|_{L^2} \lesssim 1$. Using the properties of the operators $H_{1,\eps}$ and $H_{\eps,2}$ and of $\Psif$ it follows that 
\begin{align}
(H_0 + \eps H_{1,\eps} + \eps^{2}H_{2,\eps}) \Psif = \eps^2\lambda_\eps \Psif + \eps F_\eps,
\end{align}
with $\| F_\eps \|_{L^2} \lesssim 1$. Hence, subtracting this equation to the one for $\Psi_\eps$ above and switching back to the macroscopic local coordinates, we infer by \eqref{decompose.psi} that
\begin{align}
H_\eps R_\eps = \lambda_\eps R_\eps+ \eps^{-1} F_\eps \ \ \ \ \text{in $U$}
\end{align}
We now test this equation with $\eta^2 R_\eps$, where $\eta$ is any cut-off for $\tilde U$ in $U$. This yields
\begin{align}\label{Caccioppoli}
\| (\nabla + i \eps^{-2}a_\eps) R_\eps \|_{L^2(\tilde U)}^2 \lesssim (\lambda_\eps + \eps^{-1}) \| R_\eps\|_{L^2(U)} + \|(\nabla& + i \eps^{-2}a_\eps) R_\eps\|_{L^2(U \backslash \tilde U)} + \| R_\eps \|_{L^2(U \backslash \tilde U)} .
\end{align}
Since every point in $U \backslash \tilde U$ is at distance $\sim 1$ away from the boundary, the localization estimate of Theorem \ref{t.edge}, $(a)$ for $\Psi_\eps$ and the decay of the functions $H_n$ appearing in the definition of $\Psif$ (c.f. Lemma \ref{l.branches}) imply that for every $n \in \N$
\begin{align*}
\|(\nabla& + i \eps^{-2}a_\eps) R_\eps\|_{L^2(U \backslash \tilde U)} + \| R_\eps \|_{L^2(U \backslash \tilde U)}\\
& \lesssim \|(\nabla + i \eps^{-2}a_\eps) \Psi_\eps\|_{L^2(U \backslash \tilde U)} + \| \Psi_\eps \|_{L^2(U \backslash \tilde U)} + \|(\nabla + i \eps^{-2}a_\eps) \Psif\|_{L^2(U \backslash \tilde U)} + \| \Psif \|_{L^2(U \backslash \tilde U)} \lesssim_n \eps^n.
\end{align*}
Inserting this into \eqref{Caccioppoli} implies the second inequality in \eqref{error.norm}.

\smallskip

We now insert \eqref{decompose.psi} into \eqref{def.flux} so that 
\begin{align*}
\int_{\mathop{dist}(x, \Gamma) < \eps^{\frac 2 3}} |\eps j_\eps| &\leq \eps\int_{\mathop{dist}(x, \Gamma) < \eps^{\frac 2 3}} |\Imm(\overline{\Psi}_{\eps, \textrm{flat}} (\nabla + i \eps^{-2}a_\eps) \Psifbar)|\\
&+ \eps\int_{\mathop{dist}(x, \Gamma) < \eps^{\frac 2 3}} |\Imm((\overline{\Psi}_{\eps, \textrm{flat}} + R_\eps) (\nabla + i \eps^{-2}a_\eps) R_\eps)| \\
&+ \eps\int_{\mathop{dist}(x, \Gamma) < \eps^{\frac 2 3}} |\Imm(\overline{R}_\eps (\nabla + i \eps^{-2}a_\eps)\Psif)|.
\end{align*}
By Cauchy-Schwartz inequality, the definition of $\Psif$ and \eqref{error.norm}, the last two terms vanish in the limit $\eps \to 0$. Hence
\begin{align}\label{fluxes.2}
\int_{\mathop{dist}(x, \Gamma) < \eps^{\frac 23}} |\eps j_\eps| \leq \eps\int_{\mathop{dist}(x, \Gamma) < \eps^{\frac 2 3}} |\Imm(\overline{\Psi}_{\eps, \textrm{flat}} (\nabla + i \eps^{-2}a_\eps) \Psifbar)| + o(1).
\end{align}
We establish \eqref{bound.A} from this inequality relying on the explicit formula for $\Psif$: We first remark that, thanks to Lemma \ref{l.gauge.change}, the operator $(\nabla + i \eps^{-2}a_\eps)$ may be written in curvilinear coordinates as
\begin{align*}
\nabla + i \eps^{-2}a_\eps &= \biggl(\partial_s + \frac i 2\eps^{-2} (\kappa'(\xi)\alpha(\xi) + \kappa(\xi)\alpha'(\xi))s^2\biggr) \vN + \frac{1}{1+ \kappa(\xi) s}\biggl(\partial_\xi + i \eps^{-1}\int_0^{\frac s \eps} (1 + \eps \kappa(\xi) t) b(\xi, t) \, \d t \biggr) \vT\\
& \quad\quad + i V_1 + i \eps^{-2} s^3 V_2, 
\end{align*}
with $V_1, V_2$ satisfying the bounds in Lemma \ref{l.gauge.change} and the functions $\kappa, \alpha \in C^1(\T)$ being there defined. This, together with the change of variable $s = \eps \mu$, implies that
\begin{align*}
\eps\int_{\mathop{dist}(x, \Gamma) < \eps^{\frac 2 3}}& |\Imm(\overline{\Psi}_{\eps, \textrm{flat}} (\nabla + i \eps^{-2}a_\eps) \Psifbar)| \\
&\lesssim   \eps^2 \int_\T \int_{|\mu| < \eps^{-\frac 1 3}} |\Imm(\Psifbar \frac{1}{1+ \eps \kappa(\xi)\mu}\biggl(\partial_\xi + i\eps^{-1} \int_0^{\mu}(1 + \eps \kappa(\xi) t) b(\xi, t) \d t \biggr) \Psif)|\\
& + \eps^2 \int_\T \int_{|\mu| < \eps^{-\frac 1 3}} |\Imm(\Psifbar (\eps^{-1}\partial_\mu + i \mu^2) \Psif)| \\
&+  \eps^2 \int_\T \int_{|\mu| < \eps^{-\frac 1 3}} |\Imm(\Psifbar (i V_1 + i \eps^2 \mu^2 V_2) \Psif)|
\end{align*}
Since the functions $\Psif$ may be written as $\Psif = F(\xi) G(\mu, \xi)$ with $G$ being a real-valued function, the term containing the derivative $\partial_\mu$ vanishes.  Furthermore, 
using the properties of the functions $H_1(\cdot, \cdot)$ in the definition of $\Psif$ (c.f. Lemma \ref{l.branches}), both the last term and the remaining term in the second-to-last vanish in the limit $\eps \to 0$. This yields that
\begin{align*}
\eps\int_{\mathop{dist}(x, \Gamma) < \eps^{\frac 2 3}}& |\Imm(\overline{\Psi}_{\eps, \textrm{flat}} (\nabla + i \eps^{-2}a_\eps) \Psifbar)| \\
&\lesssim   \eps^2 \int_\T \int_{|\mu| < \eps^{-\frac 1 3}} |\Imm(\Psifbar \frac{1}{1+ \eps \kappa(\xi)\mu}\biggl(\partial_\xi + i\eps^{-1} \int_0^{\mu}(1 + \eps \kappa(\xi) t) b(\xi, t) \d t \biggr) \Psif)| + o(1).
\end{align*}
Using a similar argument, and relying again on the explicit formulation of $\Psif$, on the properties of the functions $H_1(\cdot, \cdot)$ (Lemma \ref{l.branches}) and on the boundedness of the function $k_1$ (Lemma \ref{l.implicit.function} and $B_1$ (\eqref{C.ell.hard}), we may reduce the above inequality further reduces to
\begin{align*}
\eps\int_{\mathop{dist}(x, \Gamma) < \eps^{\frac 2 3}}& |\Imm(\overline{\Psi}_{\eps, \textrm{flat}} (\nabla + i \eps^{-2}a_\eps) \Psifbar)| \\
&\lesssim   \eps^2 \int_\T \int_{|\mu| < \eps^{-\frac 1 3}} |\Imm(\Psifbar \biggl(\partial_\xi + i\eps^{-1} \int_0^{\mu} b(\xi, t) \d t \biggr) \Psif)| + o(1).
\end{align*}
Using again the formulation for $\Psif$, we notice that
\begin{equation}
\begin{aligned}\label{computation.fluxes}
\Imm(\Psifbar & \biggl(\partial_\xi + i\eps^{-1} \int_0^{\mu} b(\xi, t) \d t \biggr) \Psif) \\
&= \eps^{-2}\frac{|\partial_k\nu_{1}(k_1(\xi), \xi)|^2}{\int_\T |\partial_k\nu_1(k_1(y), y)|^2 \, \d y} (k_1(\xi) + \int_0^\mu b(\xi, t) \, \d t) |H_1(k_1(\xi), \mu)|^2 \\
& \quad + \eps^{-\frac 1 2}\Imm\biggl( \Psifbar e^{\frac{i}{\eps} \int_0^\xi k_1(y) \, \d y} \partial_\xi \bigl( \frac{|\partial_k\nu_{1}(k_1(\xi), \xi)|}{\bigl(\int_\T |\partial_k\nu_1(k_1(y), y)|^2 \, \d y\bigr)^{\frac 1 2}} e^{i \int_0^\xi B_1(k_1(y),y)\, \d y}  H_1(k_1(\xi), \mu)\bigr) \biggr)
\end{aligned}
\end{equation}
Using again the regularity of the curve $k_1$ and of $B_1$ and the properties of $H_1(\cdot, \cdot)$, we infer that 
\begin{equation}
\begin{aligned}\label{error.fluxes}
 \eps^{\frac 3 2} \int_\T \int_{|\mu| < \eps^{-\frac 1 3}}& |\Imm\biggl( \Psifbar e^{\frac{i}{\eps} \int_0^\xi k_1(y) \, \d y}\\
 &\times  \partial_\xi \bigl( \frac{|\partial_k\nu_{1}(k_1(\xi), \xi)|}{\bigl(\int_\T |\partial_k\nu_1(k_1(y), y)|^2 \, \d y\bigr)^{\frac 1 2}} &e^{i \int_0^\xi B_1(k_1(y),y)\, \d y}  H_1(k_1(\xi), \mu)\bigr) \biggr) |= o(1)
\end{aligned}
\end{equation}
so that
\begin{align*}
\eps\int_{\mathop{dist}(x, \Gamma) < \eps^{\frac 2 3}}& |\Imm(\overline{\Psi}_{\eps, \textrm{flat}} (\nabla + i \eps^{-2}a_\eps) \Psif)| \\
&\lesssim   \eps^2 \int_\T \int_{|\mu| < \eps^{-\frac 1 3}} \frac{|\partial_k\nu_{1}(k_1(\xi), \xi)|^2}{\int_\T |\partial_k\nu_1(k_1(y), y)|^2 \, \d y} |k_1(\xi) + \int_0^\mu b(\xi, t) \, \d t| |H_1(k_1(\xi), \mu)|^2 + o(1).
\end{align*}
By the definition of the set $\Sigma$, the term $ \frac{|\partial_k\nu_{1}(k_1(\xi), \xi)|^2}{\int_\T |\partial_k\nu_1(k_1(y), y)|^2 \, \d y}$  is bounded in $\T$. Moreover, since  $b \in L^\infty(\R^2)$, $k_1 \in C^0(\T)$ and the function $H_1(k_1(\xi), \cdot)$ decays exponentially fast, we infer that the first term on the right-hand side is uniformly bounded. Inserting this inequality into \eqref{fluxes.2}, we establish \eqref{bound.A}.

\medskip

Equipped with \eqref{loc.A} and \eqref{bound.A}, we are now ready to prove the main statement. Let  $\rho \in \mathcal{D}(\R^2)$: Since the function $\rho$ is uniformly continuous in any compact set, \eqref{bound.A} also implies that
\begin{align*}
\lim_{\eps \to 0} \eps \int j_\eps \rho = \lim_{\eps \to 0} \eps \int_\T \rho(\xi, 0) \int_{|s| < \eps^{\frac 2 3}} j_\eps.
\end{align*}
From this identity, the proof of the corollary follows by a computation very similar to the one for \eqref{bound.A}: As in the proof of the latter, since $\rho$ is bounded in a neighbourhood of $\Gamma$, we reduce the previous identity to  
\begin{align*}
 \lim_{\eps \to 0} \eps \int_\T \rho(\xi, 0) \int_{|s| < \eps^{\frac 2 3}} j_\eps = \eps^{2}\lim_{\eps \to 0} \int_\T \rho(\xi, 0) \int_{|\mu|< \eps^{-\frac 1 3}} \Imm(\Psifbar (\partial_\xi + i \eps^{-1} \int_0^\mu b(\xi, t) \, \d t) \Psif)
\end{align*}
and, using \eqref{computation.fluxes} and \eqref{error.fluxes}, also to
\begin{equation}\label{fluxes.3}
\begin{aligned}
 \lim_{\eps \to 0} \eps \int_\T &\rho(\xi, 0) \int_{|s| < \eps^{\frac 2 3}} j_\eps\\
 & = \lim_{\eps \to 0} \int_\T \rho(\xi, 0)  \frac{|\partial_k\nu_{1}(k_1(\xi), \xi)|^2}{\int_\T |\partial_k\nu_1(k_1(y), y)|^2 \, \d y}  \int_{|\mu|< \eps^{-\frac 1 3}}  (k_1(\xi) + \int_0^\mu b(\xi, t) \, \d t) |H_1(k_1(\xi), \mu)|^2
\end{aligned}
\end{equation}
By \eqref{derivatives.nu} in Lemma \ref{l.derivatives.nu}, the term 
\begin{align}
 \int_{|\mu|< \eps^{-\frac 1 3}}  (k_1(\xi) + \int_0^\mu b(\xi, t) \, \d t) |H_1(k_1(\xi), \mu)|^2 \to \partial_{k}\nu_1(k_1(\xi), \xi).
\end{align}
This, identity  \eqref{fluxes.3}, together with the boundedness of $\rho$ and $ \frac{|\partial_k\nu_{1}(k_1(\xi), \xi)|^2}{\int_\T |\partial_k\nu_1(k_1(y), y)|^2 \, \d y}$, implies the statement of the corollary in the case $N=1$.

\medskip

In the general case $N \geq 1$, the proof of the corollary may be argued in a similar way with only a few modification when we compute the term $\Imm(\Psifbar (\partial_\xi + i \eps^{-1} \int_0^\mu b(\xi, t) \, \d t ) \Psif)$ in \eqref{computation.fluxes}. To simplify the notation, let us write 
\begin{align*}
\Psif(\xi, s) =\eps^{-\frac 1 2} \sum_{\ell=1}^N e^{\frac{i}{\eps}\int_0^\xi k_{\eps,\ell}(y) \, \d y} F_{\ell}(\xi) H_{n_\ell}(k_{\ell}(\xi), \frac s \eps), \ \ \ F_{\ell}(\xi) := A_\ell e^{i\int_0^\xi B_\ell(k_\ell(y), y) \, \d y} \frac{\partial_k\nu_{n_\ell})(k_\ell(\xi), \xi)}{\int_\T \partial_k \nu_{n_\ell}(k_{\eps,\ell}(y),y) \, \d y}.
\end{align*}
With this notation, \eqref{computation.fluxes} turns into
\begin{align*}
\Imm&(\Psifbar (\partial_\xi + i \eps^{-1} \int_0^\mu b(\xi, t) \, \d t ) \Psif) \\
&= \eps^{-1} \sum_{\ell, m =1}^N \overline{F_m}(\xi) F_\ell(\xi) e^{\frac{i}{\eps}\int_0^\xi(k_\ell(y) - k_m(y)) \, \d y} (k_\ell + \int_0^\mu b(\xi, t) \, \d t) H_\ell(k_\ell(\xi), \mu) H_m(k_m(\xi),\mu) + o(1).
\end{align*}
The proof of \eqref{bound.A} for $N \geq 1$ follows from this identity with an argument similar to the one for case $N=1$. The proof of the main statement follows again from  \eqref{bound.A} and the identity above as done for the case $N=1$ provided that
\begin{align*}
\sum_{m \neq \ell } \int_\T \rho(0,\xi)\overline{F_m}(\xi) F_\ell(\xi) e^{\frac{i}{\eps}\int_0^\xi(k_\ell(y) - k_m(y)) \, \d y}  \int_{|\mu| < \eps^{-\frac 1 3}} (k_\ell + \int_0^\mu b(\xi, t) \, \d t) H_\ell(k_\ell(\xi), \mu) H_m(k_m(\xi) \mu = o(1).
\end{align*}
This identity may be obtained by Riemann-Lebesgue Lemma {\cite[Proposition 3.2.1]{Grafakos.book}} since the function $g(\xi) := \rho(0,\xi) \overline{F_m}(\xi) F_\ell(\xi)  \int_{|\mu| < \eps^{-\frac 1 3}} (k_\ell + \int_0^\mu b(\xi, t) \, \d t) H_\ell(k_\ell(\xi), \mu) H_m(k_m(\xi) \mu \in C^1(\T)$ and $|k_\ell(\xi) - k_m(\xi)| > \delta$ for every $\xi \in \T$ (c.f. also Remark \ref{rem.resonances}).

\end{proof}

\subsection{Proof of Proposition \ref{p.asy}}
The proof of Proposition \ref{p.asy} is very similar to the one for \cite[Theorem 2.4]{GV} and we refer to \cite[Subsection 2.3]{GV} for a detailed discussion on the general strategy behind these proofs. We thus merely sketch of the steps of the argument that only require a trivial modification of the proof of \cite[Theorem 2.4]{GV} and only focus on the new parts. We stress that the main technical challenge in the current paper is the macroscopic change of the magnetic field $b_\eps$ along the curve $\Gamma$. In the blow-up analysis, the magnified asymptotic problem depends on the point of $\xi \in \Gamma$ around which we perform the blow-up: for every $\xi \in \Gamma$ as before, the limit problem resembles the standard Iwatsuka model described in Subsection \ref{sub.Iwatsuka} with magnetic field $b= b(\xi ; \cdot)$. {As the microscopic limit problems do depend on the macroscopic coordinate $\xi \in \T$, also the solution $k$ to the eikonal equation \eqref{eikonal.equation} \textit{does} depend on $\xi \in \T$.}

\vv

The proof of Proposition \ref{p.asy} relies on the analogue of \cite[Proposition 5.1]{GV} that is adapted to this setting. This means, in particular, that the exact same statements of \cite[Proposition 5.1]{GV} are true also if we replace the harmonic oscillator $\mathcal{O}(k):= - \partial_x^2 + ( x - k)^2$ with $\mathcal{O}(k, \xi)$ as defined in \eqref{harmonic.oscillator}. The proof of this result may be proven exactly as done for \cite[Proposition 5.1]{GV} if we rely on Lemma \ref{l.branches} instead of \cite[Lemma 2.1 and Lemma 5.4]{GV}. Throughout the proof below, we thus refer to \cite[Proposition 5.1]{GV} with the understanding that this holds for the operator $\mathcal{O}(k, \xi)$.

\begin{proof}[Proof of Proposition \ref{p.asy}]
We start by focussing on the case $N=1$: This means that there is only one curve $k = k_1(\xi)$ solving \eqref{eikonal}. The general case $N \geq 1$ is only technically more challenging: We may upgrade the argument for $N=1$ to any number of curves $N \in \N$ via the same adaptations used in \cite{GV} to pass from \cite[Theorem 2.4]{GV} to \cite[Theorem 2.5]{GV}. We briefly comment on this issue at the end of the proof.

\vv

We follow the same argument of \cite[Theorem 2.4]{GV}: Using the localization result of Theorem \ref{t.edge}, $(a)$, we may reduce to study \eqref{spectral.problem} for the family $(\tilde\Psi_\eps, \lambda_\eps)$ in the neighbourhood $U$ where the local coordinates are well-defined (c.f. \cite[(3.17)-(3.18)]{GV}). Thanks to this, we may apply the results of Subsection \ref{sub.local.hamiltonian} and appeal to Lemma \ref{l.local.hamiltonian} to rewrite, up to a change of gauge $\Psi_\eps \mapsto e^{i\theta_\eps} \Psi_\eps$, the equation for $e^{i\theta_\eps} \Psi_\eps$ into the microscopic coordinates $(\mu, \theta)$ introduced in \eqref{micro.coordinates}. Throughout this proof, we simplify the notation by writing $\Psi_\eps$ instead of $e^{i\theta_\eps} \Psi_\eps$.

\vv

 We define the quantity 
\begin{align}\label{m.eps}
m_\eps:= \bigl(\max_{\xi \in \T} \int_{d(\tilde\xi, \xi) < \eps}\int |\Psi_\eps|^2\bigr)^{\frac 1 2}
\end{align}
and, for every $\xi \in \T$ fixed, we consider the rescaled function
\begin{align}\label{Psi.blow.up}
\tilde \Psi_\eps(\xi, \theta, \mu):= \frac{\eps}{m_\eps}\Psi_\eps( \xi + \eps \theta , \eps\mu).
\end{align}

\begin{itemize}

\item[1.] The first step is the analogue of \cite[Proposition 3.3]{GV} and amounts to show that for every $\xi \in \T$ and $\omega << \eps^{-1}$, we have
\begin{equation}\label{asympt.approx}
\begin{aligned}
\int_{|\theta- \omega| \leq 1}\int (1 + |\mu|)^{-6} |\tilde\Psi_\eps(\xi; \theta, \mu) - e^{\frac i \eps \int_\xi^{\xi+\eps \theta} k(y) \, \d y } A_\eps(\xi) &(1 + i \eps \theta C_1(\xi)) H(\xi + \eps \theta, k(\xi+ \eps\theta), \mu)|^2 \, \d\mu \d\theta\\ &\lesssim \eps^2|\omega|^3 + \eps
\end{aligned}
\end{equation}
for some $A_\eps(\xi) \in \C$, $|A_\eps(\xi) | \lesssim 1$ and where
\begin{align}\label{C.1}
C_1(\xi) = \frac{B_1(\xi, k(\xi))}{\partial_k \nu_1(\xi, k(\xi))} + i \frac{d}{d\xi}\log\bigl(\partial_k \nu_1(\xi, k_1(\xi))\bigr), \ \ \text{ with $B_1$ is as in \eqref{C.ell.hard}.}
\end{align}

\smallskip

We prove \eqref{asympt.approx} as done in \cite[proof of Proposition 3.3, Step 1]{GV}: Using definitions \eqref{Psi.blow.up} and \eqref{m.eps}, and the equation for $\tilde \Psi_\eps$, we have that $\tilde \Psi_\eps$ is uniformly bounded in  $H^1( \{|\theta| < R\} \times \R )$ for every $R> 0$. Hence, up to a subsequence, we have that $\Psi_\eps \rightharpoonup \Psi_0$ in $H^1( \{|\theta| < R\} \times \R )$, for every $R> 0$. We thus use Lemma \ref{l.local.hamiltonian} to pass to the limi in the equation for $\tilde \Psi_\eps$ and infer that $\Psi_0(\theta, \mu) = A(\xi) e^{i k(\xi) \theta} H_1(\xi, k(\xi), \mu)$ for some $A(\xi) \in \C$ such that $|A(\xi)| \leq 1$ for every $\xi \in \T$.  

Arguing as in \cite[Proof of Proposition 3.3, Step 2]{GV}, we now proceed to consider the next-order approximation: We define the term $\Psi_{\eps,1}:= \frac{\tilde\Psi_\eps - \Psi_{\eps,0}}{\eps}$, with $\Psi_{0,\eps}= A_\eps(\xi) e^{i k_\eps(\xi) \theta} H_1(\xi, k_\eps(\xi), \mu)$ and $A_\eps(\xi) \to A(\xi)$ defined as in \cite[(3.52)]{GV}. We use again the decomposition for $H_\eps$ of Lemma \ref{l.local.hamiltonian} to write:
\begin{align}\label{equation.psi.1.eps}
(H_0 - \eps^2\lambda_\eps)\Psi_{\eps,1} = H_{1,\eps}\Psi_\eps + \eps H_{2,\eps} \Psi_\eps + f_{0,\eps} + \theta f_{1,\eps} + \eps \theta^2 f_{2,\eps}, 
\end{align}
 with $H_{1,\eps}, H_{2,\eps}$ defined as in Lemma \ref{l.local.hamiltonian} and 
 \begin{align*}
 f_{0,\eps}&:=  (\int_0^\mu \partial_\xi b(\xi; t) \, \d t )\Psi_\eps,\\
 f_{1,\eps}&:= -  2 i \bigl(\int_0^\mu \frac{b(\xi+ \eps \theta , t) - b(\xi, t)}{\eps\theta} \, d t \bigr) \bigl(\partial_\theta + i \int_0^\mu b(\xi, t) \, \d t) \\
 f_{2,\eps}&:= \bigl(\int_0^\mu \frac{b(\xi+ \eps \theta , t) - b(\xi, t)}{\eps\theta} \, \d t \bigr)^2.
 \end{align*}
 We stress that the additional terms $f_{0,\eps}, f_{1,\eps}, f_{2,\eps}$ appear in the equation since the operator $H_{0,\eps}$ in Lemma \ref{l.local.hamiltonian} does depend on the variable $\theta$. 
In contrast with \cite[(3.57) and (3.58)]{GV}, these new terms imply that the right-hand side in the previous equation grows as $\eps\theta^2 + \theta$. This yields, by  \cite[Proposition 5.1]{GV}, that the function $\Psi_{\eps,1}$ satisfies for every $R > 0$
\begin{align*}
\bigl(\fint_{|\theta|< R} \int |\Psi_{\eps,1}(\theta, \mu)| \, \d\mu \, \d\theta \bigr)^{\frac 1 2} \lesssim \eps R^3 + R^2.
\end{align*}
By the same arguments of \cite[Proof of Proposition 3.3, (3.60)]{GV}, we may pass to the (weak) limit in $\eps \to 0$. By the previous inequality, the limit function $\Psi_1$ grows at most quadratically in the variable $\theta$ and, by \eqref{equation.psi.1.eps}, it solves
 \begin{align}\label{equation.psi.1}
(H_0 - \lambda)\Psi_{1} &= H_{1}\Psi_0 + (\int_0^\mu \partial_\xi b(\xi; t) \, \d t )\Psi_0\\
&\quad\quad\quad -  2 i \theta \bigl(\int_0^\mu \partial_\xi b(\xi, t) \, d t \bigr) \bigl(\partial_\theta + i \int_0^\mu b(\xi, t) \, \d t)\Psi_0. 
\end{align}
We now want to identify $\Psi_1$, starting from the previous equation: we claim that
\begin{align}\label{identification.psi.1}
\Psi_1(\xi, \mu): = A(\xi) \biggl((i \theta C_1(\xi) + i\theta^2 \frac{k'(\xi)}{2}) H(\xi, k(\xi), \mu) +  ( \theta \frac{d}{d\xi} H(\xi, k(\xi), \mu) + W(\xi, \mu)) \biggr) e^{i k(\xi) \theta}
\end{align}
with $|C(\xi)| \lesssim 1$ and $W(\xi, \cdot) \perp H(\xi, k(\xi), \cdot)$, $\| W(\xi, \cdot) \|_{L^2} \lesssim 1$. This is the analogue of \cite[(3.54)]{GV} that is proved by appealing to \cite[Lemma 5.2]{GV}. Here, we appeal to the analogue of the previous lemma adapted to the current setting. We thus sketch below the argument: Applying the Fourier transform in the variable $\theta$ to equation \eqref{equation.psi.1}, the distribution $\hat \Psi_1= \hat \Psi_1(k , \mu)$ solves, in the sense of Schwartz distributions as in \cite[(5.28), proof of Lemma 5.2]{GV}, an equation of the form
 \begin{align}
(\mathcal{O}(k) - \lambda)\Psi_{1} &= f_1(\mu) \delta(k - k(\xi)) +  f_2(\mu) \partial_k \delta(k - k(\xi)).
\end{align}
Here, the functions $f_1, f_2 : \R \to \R$ are defined as
\begin{align*}
f_1(\mu)&:=\biggl( - \kappa(\xi) \partial_\mu - 2 \kappa(\xi) \mu |k(\xi)|^2 - 2 \kappa(\xi) (\int_0^\mu (\mu- t) b \, \d t ) k(\xi) - i (3 \alpha'(\xi) \kappa(\xi) + \alpha(\xi) \kappa'(\xi)) \mu^2 \partial_\mu^2 \\
&\quad - i  (3 \alpha'(\xi) \kappa(\xi) + \alpha(\xi) \kappa'(\xi))\mu  - 2\kappa(\xi) (\int_0^\mu (\mu- t) b(\xi, t) \, \d t )(\int_0^\mu  b(\xi, t) \, \d t) \biggr)H_1(\xi, k(\xi), \mu)\\
& \quad\quad + i (\int_0^\mu \partial_\xi b(\xi, t) \, \d t)H_1(\xi, k(\xi), \mu)\\
f_2(\mu)&:= - 2 i \bigl(\int_0^\mu \partial_\xi b(\xi, t) \, \d t \bigr)\bigl( \int_0^\mu b(\xi, t) \, \d t - k(\xi)\bigr)H_1(\xi, k(\xi), \mu).
\end{align*}
We recall that here the function $\kappa$ denotes the curvature of $\Gamma$ and $\alpha$ is defined as in \eqref{Neumann.to.Dirichlet}.
Using the same argument in \cite[Proof of Lemma 5.2]{GV} for $\Psi_1$, we infer that
\begin{equation}\label{hat.psi.1}
\begin{aligned}
\hat \Psi_1(k,\mu) &:=C_2 H(\xi, k(\xi), \mu) \partial_k^2 \delta( k - k(\xi)) + (C_1 H(\xi, k(\xi), \mu) + W_2(\mu)) \partial_k \delta(k - k(\xi))\\
& \quad\quad\quad +  W_1  \delta(k - k(\xi))
\end{aligned}
\end{equation}
with $C_2$ and $W_2$ solving
\begin{equation}
\begin{aligned}\label{C.2.term}
&\int f_2(\mu) H(\xi, k(\xi), \mu) \, \d\mu - 4 C_2 \int (k(\xi)- \int_0^\mu b(\xi, t) \, \d t) |H(\xi, k(\xi), \mu)|^2 \, \d\mu = 0,\\
&(O(k(\xi)) - \lambda)W_2 =  f_2 - 4 C_2 (k(\xi)- \int_0^\mu b(\xi, t) \, \d t) H(\xi, k(\xi), \mu) \ \ \ \text{in $\R$,} \ \ \ W_2 \perp H(\xi, k(\xi), \cdot)
\end{aligned}
\end{equation}
and $C_1$ such that
\begin{equation}\label{C.1.term}
\begin{aligned}
\int f_1(\mu) H(\xi, k(\xi), \mu) \, \d\mu& - 4 C_1 \int (k(\xi)- \int_0^\mu b(\xi, t) \, \d t) |H(\xi, k(\xi), \mu)|^2\, \d\mu\\
& \quad\quad -  2 \int (k(\xi)- \int_0^\mu b(\xi, t) \, \d t) H(\xi, k(\xi), \mu) W_2(\mu) \, \d\mu - 2 C_2= 0.
\end{aligned}
\end{equation}
We stress that the term $\delta(k - k(\xi))$ in \eqref{hat.psi.1} does not contain a term proportional to $H_1(\xi, k(\xi), \cdot)$ due to the choice of $A_\eps$ (see \cite[(3.52)]{GV} and \cite[Second identity in (3.64)]{GV}).

Using the first identity in \eqref{C.2.term} and formula \eqref{derivatives.nu} in Lemma \ref{l.derivatives.nu}, we infer that
\begin{align}\label{C.2}
C_2 = \frac i 2 k'(\xi).
\end{align}
Inserting this into the second equation in \eqref{C.2.term}, and appealing to \eqref{deriv.H} of Lemma \ref{l.derivatives.nu}, we get that
\begin{align}\label{W.2}
W_2(\mu) = i \frac{d}{d\xi}H(\xi, k(\xi), \mu).
\end{align}
Finally, the previous two formulas, \eqref{C.1.term} and Lemma \ref{l.derivatives.nu} yield that $C_1$ is as in \eqref{C.1}. Inserting \eqref{C.2}, \eqref{W.2} and \eqref{C.1} into \eqref{hat.psi.1}, we conclude that
\begin{align*}
\Psi_1(\xi, \mu): = A(\xi) \biggl((i \theta C_1(\xi) + i\theta^2 \frac{k'(\xi)}{2}) H(\xi, k(\xi), \mu) +  ( \theta \frac{d}{d\xi} H(\xi, k(\xi), \mu) + W(\xi, \mu)) \biggr) e^{i k(\xi) \theta}
\end{align*}
with $W(\xi, \cdot) \perp H(\xi, k(\xi), \cdot)$, $\| W(\xi, \cdot) \|_{L^2} \lesssim 1$.

\smallskip

As in \cite[Proof of Proposition 3.3, Step 3]{GV} we now turn to the second-order approximation and consider the term $\Psi_{2,\eps}:= \frac{\Psi_{1,\eps}- \Psi_{\eps,1}}{\eps}$. As for $\Psi_{1,\eps}$ above, also in this case the change in the magnetic field produces new terms on the right hand side that have a higher growth in $\theta$ with respect to the analogue problem in \cite[(3.67) and display above]{GV}. This, in particular implies that 
\begin{align*}
\bigl(\fint_{|\theta|< R} \int |\Psi_{\eps,2}(\theta, \mu)| \, \d\mu \, \d\theta \bigr)^{\frac 1 2} \lesssim \eps R^4 + R^3.
\end{align*}
Spelling out the definition of $\Psi_{\eps,2}$ and using the properties of the exponential, this inequality yields, in turn, inequality \eqref{asympt.approx}.

\medskip

\item[2. ] We now claim that the $F_\eps(\cdot) := e^{-\frac i \eps \int_0^{\cdot} k(y) \, \d y} A_\eps(\cdot)$ satisfies for every $\xi \in \T$ and $|\omega| > 1$ such that $|\eps \omega | << 1$ the inequality
\begin{align}\label{ascoli.arzela}
| F_\eps( \xi + \eps \omega) - F_\eps(\xi)( 1 + i\eps \omega \, C_1(\xi))| \lesssim \eps + \eps^2 |\omega|^3.
\end{align}
The term $C_1$ is the same as in \eqref{asympt.approx}. We remark that, with this definition of $F_\eps$, inequality \eqref{asympt.approx} may be rewritten as
\begin{equation}\label{asympt.approx.2}
\begin{aligned}
\int_{|\theta- \omega| \leq 1}\int (1 + |\mu|)^{-6} |\tilde\Psi_\eps(\xi; \theta, \mu) - e^{\frac i \eps \int_0^{\xi+\eps \theta} k(y) \, \d y } F_\eps(\xi) &(1 + i \eps \theta C_1(\xi)) H(\xi + \eps \theta, k(\xi+ \eps\theta), \mu)|^2\\ &\lesssim \eps^2|\omega|^3 + \eps.
\end{aligned}
\end{equation}
The proof of \eqref{ascoli.arzela} is very similar to the one for \cite[(3.34), proof of Lemma 3.4]{GV}: We remark that the function 
\begin{align}\label{amplitude.approx}
f(\xi; \eps\omega):= \int_{|\theta - \omega| < 1} \int \tilde \Psi_\eps(\xi, \theta, \mu) e^{- \frac{i}{\eps}\int_\xi^{\xi+\eps\theta} k(y) \, \d y} H_1(\xi+ \eps \theta, k(\xi+ \eps\theta), \mu) \, \d\theta \, \d\mu
\end{align}
satisfies
\begin{align*}
&|f(\xi, \eps \omega) - A(\xi)( 1 + C_1(\xi) \eps \omega)| \lesssim \eps + \eps^2|\omega|^3, \ \ \ \ |f(\xi, 0) - A(\xi)| \lesssim \eps\\
&f(\xi, \eps \omega) = e^{-\frac{i}{\eps}\int_\xi^{\xi+\eps\omega} k} f(\xi+\eps\omega, 0).
\end{align*}
The first inequality above follows directly from \eqref{asympt.approx} and the properties of the eigenfunctions $H_1(\xi, k(\xi), \cdot)$; the second inequality is a consequence of the first with the choice $\omega= 0$. The third inequality is a simple change of variables in the definition of $f(\cdot, \cdot)$. Wrapping together the previous inequalities yields \eqref{ascoli.arzela}.

\medskip

\item[3.]  We now show that $F_\eps \to F$ uniformly on $\T$ with
$$
F' = i C_1 F \ \ \ \ \text{in $\T$}, \ \ \ \ |F(0)| = 1.
$$
This implies that every limit function $F$ satisfies
$$
F := A_1 \frac{\partial_k \nu_1(\xi, k_1(\xi))}{\partial_k\nu_1(0, k_1(0))} e^{i \int_0^\xi \frac{\partial_k\nu_1(y, k_1(y))}{B_1(y)} \, \d y},
$$
for some $A_1 \in \C$ with $|A_1| =1$. The previous ODE follows from Step 2 by an argument similar to the one in \cite[Proof of Lemma 3.4]{GV} using Ascoli-Arzela's theorem and \cite[(3.33)]{GV}: The main difference, in this case, is that the increments cannot immediately be chosen to be macroscopic (i.e. $\omega \sim \eps^{-1}$).  We first claim that $C_1'$ is continuous on $\T$ and hence that the function $C_1: \T \to \R$ is Lipschitz. The continuity of $C_1'$ is a simple consequence of the definition \eqref{asympt.approx}, the regularity of the functions $\nu_1(\cdot, \cdot)$ and $k_1$ (c.f.  Lemma \ref{l.branches} and Lemma \ref{l.implicit.function}) and the assumption $\lambda \in \Sigma$ that implies that $\partial_k \nu_1(\xi, k(\xi)) > \epsilon$ for every $\xi \in \T$ and for some $\epsilon >0$.

\smallskip

We now define $\delta_\eps := \eps \omega$ and remark that, whenever $\eps << \delta_\eps << \eps^{\frac 1 2}$, inequality \eqref{ascoli.arzela} may be rewritten as
\begin{align*}
| F_\eps( \xi + \delta_\eps) - F_\eps(\xi)( 1 + \delta_\eps C_1(\xi))| \lesssim o(\delta_\eps), \ \ \ \text{for every $\xi \in \T$, $\delta_\eps$ as above.}
\end{align*}
Since $C_1(\cdot)$ is Lipschitz, using a telescopic sum the inequality above implies that also 
\begin{align*}
| F_\eps( \xi + \delta) - F_\eps(\xi)( 1 + \delta C_1(\xi))| \lesssim o(\delta) , \ \ \ \text{for every $\xi \in \T$, $\delta \sim 1$.}
\end{align*}
From this, the argument follows as in \cite[Proof of Lemma 3.4]{GV}.

\medskip

\item[4.] Conclusion. From inequality \eqref{asympt.approx.2}  and an argument similar to the one to pass from \cite[(3.31)]{GV} to \cite[(3.22)]{GV}, we infer that for every $\xi^*\in \T$
\begin{equation*}
\begin{aligned}
\int_{|\theta| \leq 1}\int |\tilde\Psi_\eps(\xi^*; \theta, \mu) - e^{\frac i \eps \int_0^{\xi^*+\eps \theta} k(y) \, \d y } F_\eps(\xi^*) H(\xi^* + \eps \theta, k(\xi^*+ \eps\theta), \mu)|^2 \lesssim \eps.
\end{aligned}
\end{equation*}
In this case, to we get rid of the weight $(1 + |\mu|)^{-6}$  using properties of the eigenfunctions $H_n(\xi, k, \cdot)$ in Lemma \ref{l.branches}. We now combine the previous inequality with Step 3 to infer that also 
\begin{equation*}
\begin{aligned}
\int_{|\theta| \leq 1}\int |\tilde\Psi_\eps(\xi^*; \theta, \mu) - e^{\frac i \eps \int_0^{\xi^*+\eps \theta} k(y) \, \d y } F_\eps(\xi^*) H(\xi^* + \eps \theta, k(\xi^*+ \eps\theta), \mu)|^2 = o(1),
\end{aligned}
\end{equation*}
 where we abuse notation writing $\eps$ instead of a sequence $\{\eps_j\}_{j\in\N}$ and where $o(1)$ does not depend on the point $\xi_*\in \T$ but it might depend on $\{ \eps_j \}_{j\in \N}$. 

\smallskip

Using the definition of $\tilde\Psi_\eps$ and of the limit function $F$ of Step 3,  we infer, after a change of coordinates, that
\begin{equation}\label{almost.there}
\begin{aligned}
\int_{|\xi-\xi_*| \leq \eps }\int |m_\eps^{-1}\Psi_\eps - A_1 e^{\frac i \eps \int_0^{\xi} k(y) \, \d y + i \int_0^\xi \frac{B(y)}{\partial_k \nu(y, k(y))}\, \d y} \frac{\partial_k \nu(\xi_*, k(\xi_*))}{\partial_k\nu(0, k(0))} H(\xi, k(\xi), \frac s \eps)|^2 \, \d\xi \, \d s = o(1)
\end{aligned}
\end{equation}
for $A_1 \in \C$ such that $|A_1|=1$. To conclude the proof of Proposition \ref{p.asy} it thus remains to show that 
\begin{align}\label{m.eps.scaling}
\eps^{-\frac 1 2}m_\eps \to \frac{|\partial_k\nu(0, k(0))|}{\bigl(\fint_\T |\nu(y, k(y))|^2 \, \d y\bigr)^{\frac 1 2}}.
\end{align}

\smallskip

We show \eqref{m.eps.scaling} as follows: From \eqref{almost.there}, for every $\xi_* \in \T$ the triangle inequality, the fact that $|A_1|=1$ and the normalization of the eigenfunctions $H(\xi, k(\xi), \cdot)$ yield that
\begin{align*}
\int_{|\xi-\xi_*| \leq \eps }\int |\Psi_\eps|^2 = m_\eps^2 \bigl(|\frac{\partial_k \nu(\xi_*, k(\xi_*))}{\partial_k\nu(0, k(0))}|^2 + o(1)\bigr).
\end{align*}
Since $\| \Psi_\eps \|_{L^2(\R^2)} =1$, we may now find $n_\eps$ intervals $\{ I_{i,\eps}\}_{i=1}^{n_\eps}$ of size $\eps$ centred at points $\{ \xi_i^\eps \}_{i=1}^{n_\eps}$ such that 
$$
\eps n_\eps \to 1, \ \ \ \ \ S_\eps:=\sum_{i=i}^{n_\eps} \| \Psi_\eps \|_{L^2( I_{\eps,i} \times \R)} \to 1.
$$
Using this construction and the identity two displays above, we infer that
\begin{align*}
\eps^{-1} m_\eps^2 S_\eps = \bigl( \eps \sum_{i=1}^{n_\eps}|\frac{\partial_k \nu(\xi_i, k(\xi_i))}{\partial_k\nu(0, k(0))}|^2 + o(1)\bigr)^{-1}.
\end{align*}
Using that the functions $\partial_k \nu(\cdot, k(\cdot))$ are continuous and differentiable (c.f. Lemma \ref{l.branches} and Lemma \ref{l.implicit.function}), we know that
\begin{align*}
\eps \sum_{i=1}^{n_\eps}|\frac{\partial_k \nu(\xi_i, k(\xi_i))}{\partial_k\nu(0, k(0))}|^2 \to \frac{\fint_\T |\partial_k \nu(y, k(y))|^2 \, \d y}{|\partial_k\nu(0, k(0))|^2}
\end{align*}
and hence also \eqref{m.eps.scaling}. This establishes Proposition \ref{p.asy} when $N=1$.
\end{itemize}

\medskip

{We conclude by quickly remarking on the case $N > 1$, namely if the solutions to \eqref{eikonal} are more than one. This case may be treated as is done in \cite[Theorem 2.5 and Proposition 4.1]{GV} by relying on the fact that, by Lemma \ref{l.branches}, we have that 
$$
\fint_{|\xi| < C\eps} e^{\frac i \eps \int_0^\xi (k_i(s)- k_j(s)) \, \d s} \lesssim C^{-1}.
$$
The proof of Step 1 may be adapted to this case as done in \cite[Proof of Proposition 4.1]{GV}. In this case, for every $\xi \in \T$ the blow-up limit is of the form $\Psi_0(\xi; \theta, \mu)= \sum_{j=1}^N A_j(\xi) e^{\frac i \eps \int_0^\theta k_{j}(s) \, \d s} H_j(\xi, k_j(\xi), \mu)$ for $A_1(\xi), \cdots, A_j(\xi) \in \C$. We stress that, in this case, in the analogue of \eqref{asympt.approx} the right-hand side contains also a term of the form $C\eps$.}

\smallskip

Step 2 may be argued similarly for each function of the form $F_{j, C, \eps}(\cdot) := e^{-\frac i \eps \int_0^{\cdot} k_{j,\eps}(y) \, \d y} A_j(\cdot)$, $j=1, \cdots, N$. The argument is an adaptation of the one above and \cite[Proof of Proposition 4.1, Step 2]{GV}. We stress that, in this instance, it is convenient to replace the function in \eqref{amplitude.approx} with 
\begin{align}\label{amplitude.approx}
f_{j,C}(\xi; \eps\omega):= \fint_{|\theta - \omega| < C} \eta(\frac{\theta- \omega}{C}) \int \tilde \Psi_\eps(\xi, \theta, \mu) e^{- \frac{i}{\eps}\int_\xi^{\xi+\eps\theta} k_{\eps,j}(y) \, \d y} H_j(\xi+ \eps \theta, k_{\eps,j}(\xi+ \eps\theta), \mu) \, \d\theta \, \d\mu
\end{align}
where $\eta$ is any smooth cut-off function for $\{|\theta| < 1 \}$ in $\{ |\theta| < 2\}$ and the index $j=1, \cdots ,N$. This choice, indeed, yields that the products of two different waves satisfy for every $n\in \N$ and $i, j =1, \cdots, N$, $i\neq j$
\begin{align}
\bigl| \fint_{|\xi| < \eps C} \eta(\frac{\xi}{\eps C})e^{\frac i \eps \int_0^\xi (k_j(s)- k_i(s)) \, \d s} \bigr| \lesssim_n \frac{\eps^n}{C}.
\end{align}

\smallskip

Step 3 and 4 may be argued as in the case $N=1$ if we choose an appropriate sequence $C_\eps \to +\infty$ such that $\eps C_\eps \ll \eps^{\frac 1 2}$ and set $r_\eps:= \eps C_\eps$. We stress that, in the argument for Step 3, the choice $r_\eps \ll \eps^{\frac 1 2}$ is crucial for the argument of Step 3 to work also in this setting.
\end{proof}

\appendix
\section{Appendix}
\subsection{The flat-boundary case}\label{sub.Iwatsuka}
We use the notation $x=(x_1, x_2) \in \R^2$. In this subsection we discuss and enumerate some well-known results for the spectrum and (generalized) eigenfunctions in the case of a perpendicular magnetic field $b e_3$ where the intensity $b(x)=b(x_1)$ depends only on one variable and satisfies assumptions $(A1)$-$(A3)$ with the variable $s$ replaced by $x_1$.  In this setting, we may choose a suitable gauge such that the corresponding Hamiltonian takes the form
\begin{align}\label{H.iwatsuka}
\Hiwa:= -\partial_{x_1}^2 - (\partial_{x_2} + i \int_0^{x_1} b(t) \d t)^2 \ \ \ \text{in $\R^2$}
\end{align}
and the spectrum may decomposed as 
\begin{align}\label{def.branches}
\sigma( \Hiwa):=\overline{\{ \nu_n(k) \, \colon \,  k\in \R \}}
\end{align}
where, for each $k \in \R$ fixed, the sequence $\{ \nu_l(k) \}_{n\in \N} \subset \R_+$ corresponds to the (simple) eigenvalues of the one-dimensional operator
\begin{align}\label{harmonic.oscillator}
O_{b}(k):= -\partial_{x_1}^2 + (\int_0^{x_1} b(s) \d s - k)^2 \ \ \ \ \text{in $L^2(\R)$.}
\end{align}
We remark that by assumption $(A3)$ the potential $(\int_0^{x_1} b(s) \d s - k)^2$ grows at infinity. Therefore, the spectrum of $O_{b}(k)$ is discrete with simple eigenvalues and  $\{ \nu_l(k) \}_{n\in \N}$ is a positive and increasing sequence.

\smallskip

The following lemma summarizes some well-known properties for $\sigma(\Hiwa)$ and $\sigma(O_{b}(k))$, for $k \in \R$:

\begin{lem}\label{l.branches}
\begin{itemize}
\item[(i)] For every $n\in \N$, the function $\nu_n \in C^\infty(\R)$ satisfies
\begin{align}
\nu_n \geq m (n + \frac 1 2) \ \ \ \ \ \ \lim_{k \to - \infty} \nu_n(k) = b_-( n+ \frac 12) \ \ \ \ \ \  \lim_{k \to + \infty} \nu_n(k)  = b_{+}(n + \frac 1 2 ).
\end{align}
Here, the constant $m \geq 0$ is as in $(A3)$ for $b$. Moreover, if $b$ is monotone, then each $\nu_n$ is monotone as well.
\item[(ii)] For every $n\in \N$ and $k\in \R$, let $H_n(k, \cdot) \in L^2( \R)$  be the (normalized) eigenfunction corresponding to the eigenvalue $\nu_n(k)$ for $O(k)$. Then, there exists a constant $C= C(n, m, \| b \|_{L^\infty(\R^2)})$ such that for every $R \geq 1$
\begin{align*}
\int_{|x - k| > R} |H_n(k ; x_1)|^2 \d x_1 \leq  \exp\bigl({- c(n)R}\bigr).
\end{align*}
\end{itemize}
\end{lem}

\begin{proof}
The proof of $(i)$ is standard and may be easily seen by diagonalizing the operator $\Hiwa$ and relying on the assumptions on $b$ to study the eigenvalues $\nu_n(k)$, $k \in \R$, for $O_b(k)$. Since $b$ is strictly positive, for every $k \in \N$ we can always write $(k - \int_0^{x_1} b(s) \, \d s) = \int_x^{x(k)} b(s) \, \d s $ for a unique $x(k) \in \R$. This yields that $( k- \int_0^{x_1} b(s) \, \d s)^2 \geq m^2 (x- x(k))^2$ and that, if $|k| \to \pm\infty$ also $x_k \to \pm \infty$ and thus  $\int_{x(k)}^x b(s) \, \d s \sim b_{\pm}(x - x(k))$, whenever $|x - x(k)| \leq |x(k)| -M$ with $M$ as in (A2). The proof of $(i)$ may be thus  argued via standard comparison principles and Min-Max techniques for semibounded operators (c.f. \cite[Theorem XIII.1]{ReedSimon}). 

\smallskip

The proof of $(ii)$ is a consequence of the equation solved by each function $H_n(k, \cdot)$. For the detailed proof, we refer to \cite[Lemma 3.5]{Iwatsuka}. We stress that in \cite{Iwatsuka} it is assumed that $b \in C^\infty(\R)$. The proof of \cite[Lemma 3.5]{Iwatsuka}, however, works also under the assumptions of the current section.
\end{proof}

\subsection{Local Hamiltonian}\label{sub.local.hamiltonian}
Throughout this subsection, we work in the tubular neighbourhood $U$ of the curve $\Gamma$ where the curvilinear coordinates defined in \eqref{local.coordinates} are well defined. In this section, we prove that the Hamiltonian $H_\eps$ admits, up to a change of gauge, a suitable local representation in $U$ that will prove to be useful in the blow-up analysis performed in Proposition \ref{p.asy}. 

\smallskip

Let $b_\eps$ be as in Section \ref{s.general} and let $H_\eps$ be as in \eqref{hamiltonian.intro}.  We consider as magnetic potential the vector field $\eps^{-2}a_\eps$, with $a_\eps :=(\nabla \phi_\eps)^T$ such that
\begin{align}\label{stream.function}
\begin{cases}
-\Delta \phi_\eps = b_\eps \ \ \ \text{in $\R^2$}\\
\limsup_{|x| \to \infty} |x|^{-2}|\phi_\eps(x)| < +\infty. 
\end{cases}
\end{align}
Since $b_\eps \in L^\infty(\R^2)$, it follows that $\phi_\eps \in W^{2,p}_{loc}(\R^2)$, $p \in [1; +\infty)$ by standard elliptic regularity theory and Calderon-Zygmund estimates {\cite[Theorem 9.9]{Gilbarg.Trudinger}.} 

\smallskip

If $(\xi, s)$ are  as in \eqref{local.coordinates}, for a fixed point $\xi_*\in \T$ we define the microscopic coordinates $(\theta, \mu)$
\begin{align}\label{micro.coordinates}
s \mapsto \eps \mu,  \ \ \ \ \ \xi \mapsto \xi_* + \eps \theta.
\end{align}

\smallskip

Finally, if $\Omega \subset \R^2$ is the bounded set having boundary $\Gamma$, we define the function
\begin{align}\label{Neumann.to.Dirichlet}
\alpha : \T \to \R \, \, \,\,\, \alpha(\xi):= \partial_n G(\xi, s) 
\end{align}
where $G$ is the Green function for the set $\Omega$, namely the (weak) solution to
\begin{align}\label{green.omega}
\begin{cases}
-\Delta G = \delta(x) \ \ \ &\text{in $\Omega$}\\
G= 0 \ \ \ &\text{on $\Gamma$.}
\end{cases}
\end{align}

\begin{lem}\label{l.local.hamiltonian}
For $\xi_* \in \T$, let $(\theta, \mu)$ be as in \eqref{micro.coordinates}. Let $u \in H^2(U)$ and $f \in L^2(U)$ and such that
\begin{align}
H_\eps u = f \ \ \ \ \text{in $U$.}
\end{align}
Then, there exists a global change of gauge such that $u= u(\theta, \mu)$ and $f=f(\theta, \mu)$ satisfy
\begin{align}
(H_{0,\eps} + \eps H_{1,\eps} + \eps^2 H_{2,\eps})u = \eps^2 f,
\end{align}
with the operators
\begin{equation}
\begin{aligned}
H_{0,\eps}& := - \partial_\mu^2 - (\partial_\theta + i \int_0^\mu b \, \d t )^2\\ 
H_{1,\eps} &:= - \kappa \partial_\mu + 2 \kappa \mu \partial_\theta^2 + 2 i \kappa (\int_0^\mu (\mu- t) b \, \d t ) \partial_\theta - i (3 \alpha' \kappa + \alpha \kappa') \mu^2 \partial_\mu^2 \\
&\quad - i  (3 \alpha' \kappa + \alpha \kappa')\mu  - 2\kappa (\int_0^\mu (\mu- t) b \, \d t )(\int_0^\mu  b \, \d t)
\end{aligned}
\end{equation}
and $H_{2,\eps}: H^2(U)\cap H^1_0(U) \to L^2(U)$ satisfying for every $\rho \in H^2(U)\cap H^1_0(U)$
\begin{align}\label{boundedness.H.2}
 \|H_{2,\eps} \rho \|_{L^2(U)}& \lesssim \|(1 +|\mu|) (\partial_\theta + i \int_0^\mu b )^2 \rho \|_{L^2(U)}+ \|(1 +|\mu|)^3 (\partial_\theta + i \int_0^\mu b) \rho \|_{L^2(U)}\notag \\
 &\quad \quad + \|(1 + |\mu|)^3\partial_\mu \rho\|_{L^2(U)}+ \|(1 + |\mu|)^4 \rho \|_{L^2(U)} + \eps^2\|(1 + |\mu|)^6 \rho \|_{L^2(U)}.
\end{align}
We recall that the function $\kappa$ is the curvature of $\Gamma$ and that $\alpha$ is defined as in \eqref{Neumann.to.Dirichlet}. Furthermore, we stress that the functions  $\alpha, \kappa, \alpha', \kappa'$ are evaluated at $\xi_* + \eps \theta$, while the function $b$ in the integrals is evaluated in $(\xi_* + \eps \theta, t)$.
\end{lem}

The previous result follows from:

\begin{lem}\label{l.gauge.change}
Let $u, f, \kappa, \alpha$ be as in Lemma \ref{l.local.hamiltonian}. Then, there exists a gauge $\rho_\eps \in C^2(\R^2)$ such that $\tilde u = e^{i\rho_\eps} u$ and $\tilde f :=  e^{i\rho_\eps} f $ solve
\begin{align}
-(\nabla + i \frac{\aloc}{\eps^2})\cdot (\nabla + i \frac{\aloc}{\eps^2})\tilde u = \tilde f \ \ \ \text{in $U$},
\end{align}
where the vector field $\aloc$ satisfies
\begin{align}\label{a.local}
\aloc = a_{0,\eps} + \frac 1 2 (3 \kappa \alpha' + \kappa' \alpha) s^2 \vN + \eps^2 V_1+ s^3 V_2 \ \ \ \text{in $U$}
\end{align}
with
\begin{equation}
\begin{aligned}\label{nice.a}
&a_{0,\eps}(\xi, s) = \frac{\eps}{1+\kappa(\xi) s}\biggr( \int_0^{\frac s \eps} (1 + \eps \kappa(\xi) t) b(\xi, t) \, \d t \biggr)\vT
\end{aligned} 
\end{equation}
and the error terms $V_1$ and $V_2$ such that
\begin{align*}
\|V_1\|_{L^\infty(U)} + \|V_2 \|_{L^\infty(U)} \lesssim 1.
 \end{align*}
\end{lem}

This lemma, in turn, is a consequence of the following  simple result:
\begin{lem}\label{l.extension}
Let $E \subset \R^2$ be a simply connected and open set having $C^1$ boundaries. Let $K \subset E$ be compact, simply connected and with $C^1$ boundary. Let $A, \tilde A \in H^1(E \backslash K ; \R^2)$ be two vector fields such that
\begin{align}
 \nabla \cdot A = \nabla \cdot \tilde A \ \ \ \text{in $E \backslash K$,}\ \ \ \ \int_{\partial E} A \cdot \nu = \int_{\partial E} \tilde A \cdot \nu.
\end{align}
Then there exist two extensions $V, \tilde V \in L^2(E; \R^2)$ of $A$ and $\tilde A$, respectively, such that
\begin{align}
\nabla \cdot V = \nabla \cdot \tilde V \ \ \ \text{in $E$.}
\end{align}
\end{lem}

\begin{proof}[Proof of Lemma \ref{l.local.hamiltonian}]
The proof of this result follows from Lemma \ref{l.gauge.change} and the formulation of the operator into curvilinear coordinates (see also \cite[Proof of Lemma 3.2]{GV}).
\end{proof}

\begin{proof}[Proof of Lemma \ref{l.gauge.change}]
We recall the definition of the vector field $a_\eps = (\nabla \phi_\eps)^T$, with $\phi_\eps$ solving \eqref{stream.function}. For $a_{0,\eps}$ as in \eqref{nice.a}, we define the vector field
\begin{align}\label{vector.fields}
\bar A_\eps := a_{0,\eps}^T + c_\eps \nabla G,
\end{align}
where the constant $c_\eps \in \R$ chosen such that
\begin{align*}
\int_{\Gamma} \bar A_\eps \cdot \nu = \int_\Gamma (a_\eps)^T \cdot \nu = \int_\Gamma \partial_n \phi_\eps.
\end{align*}
Since, by construction, $\nabla \times a_{0,\eps} = b_\eps$ in $U$, we have that also  $\nabla \cdot \bar A_\eps = b_\eps$ in $U$. By  \eqref{stream.function}, this also yields that $\nabla \cdot A_\eps = \nabla \cdot \bar A_\eps$ in $U$. We may thus apply Lemma \ref{l.extension}: Let $\Omega \subset \R^2$ be the bounded domain having boundary $\Gamma$. We define the sets $E = \Omega \cup U$ and $K= \overline{\Omega \backslash U}$ and apply Lemma \ref{l.extension} to the vector fields $A = a_\eps$ and $\tilde A =\bar A_\eps$. This yields that we may extend the previous fields to the whole set $E$ in such a way that 
\begin{align}\label{same.div}
\nabla \cdot \bar A_\eps = \nabla \cdot a_\eps \ \ \ \ \text{in $E$.}
\end{align}

\smallskip

Since $E$ is simply connected, the previous identity implies that there exists a function $\rho_\eps \in C^2(E)$ such that $(\bar A_\eps)^T = (A_\eps)^T + \nabla \rho_\eps$. The definition of the Hamiltonian $H_\eps$ implies that the functions $\tilde u = e^{i\eps^{-2}\rho_\eps}u, \tilde f = e^{i\eps^{-2}\rho_\eps} f$ satisfy the equation 
\begin{align}
-(\nabla + i \frac{\tilde A_\eps^T}{\eps^2})\cdot (\nabla + i \frac{\tilde A_\eps^T}{\eps^2}) \tilde u = \tilde f \ \ \ \ \text{in $U$.}
\end{align}

\bigskip

To conclude the proof of the lemma, we need to show that there exists another change of gauge which allows to replace the vector potential $\tilde A^T$ with $\aloc$ as in \eqref{a.local}. We do this by arguing as done in \cite[Lemma 3.1]{GV}: Using the local coordinates and equation \eqref{green.omega}, the regularity of $\Omega$ yields that for every $(\xi, s) \in U $ it holds
$$
G(\xi, s) = \alpha s  + \beta s^2 + \gamma s^3  + O(|s|^4)
$$ 
with $\beta = -\frac 1 2 \kappa \alpha$ and $\gamma =  \frac{1}{6}(\kappa \alpha + \kappa^2 \alpha - \alpha'')$. From this, we infer that
$$
\tilde A^T = a_0 + (\alpha + 2\beta s + 3 \gamma s^2) \vT - \frac{1}{1+\kappa s} (\alpha' s + \beta' s^2 +\gamma' s^3 ) \vN + O(s^4). 
$$
We argue as done in {\cite[proof of Lemma 3.1]{GV}} and use the definitions of $\alpha, \beta$ and $\gamma$ to set the gauge 
$$
\rho_\eps = \int_0^\xi \alpha - (\frac{|\Omega|}{(2\pi \eps)^{2}}- \lfloor \frac{|\Omega|}{(2\pi \eps)^{2}} \rfloor) 2\pi \eps^2 \xi + \frac 1 2 \alpha' s^2 
$$
such that $\aloc := \tilde A^T  + \nabla \rho_\eps $ satisfies \eqref{a.local}. We stress that both $\nabla \rho_\eps$ and the exponential function $e^{i \eps^{-2}\rho_\eps}$ are periodic in the variable $\xi \in \T$ (c.f. \cite[End of proof of Lemma 3.1]{GV}).
\end{proof}
\begin{proof}[Proof of Lemma \ref{l.extension}]
The proof of this result is a simple application of the divergence theorem: Since $A$ and $\tilde A$ have the same flux through $\partial E$ and same divergence in $E \backslash K$, it follows that also
\begin{align}\label{same.flux}
\int_{\partial K} A \cdot \nu = \int_{\partial K} \tilde A \cdot \nu.
\end{align}
We thus define the extensions $V$ and $\tilde V$ as the vector fields
\begin{align*}
V := \begin{cases}
A \ \ \ \text{in $E \backslash K$}\\
\nabla u \ \ \ \text{in $K$}
\end{cases}\ \ \ \ \tilde V := \begin{cases}
\tilde A \ \ \ \text{in $E \backslash K$}\\
\nabla \tilde u \ \ \ \text{in $K$}
\end{cases}
\end{align*}
where $u, \tilde u$ solve the Neumann problems
\begin{align*}
\begin{cases}
-\Delta u = f \ \ \ \text{ in $K$}\\
\partial_n u = A \cdot \nu \ \ \ \text{on $\partial K$}
\end{cases} \ \ \ \ \begin{cases}
-\Delta u = f \ \ \ \text{ in $K$}\\
\partial_n u = \tilde A \cdot \nu \ \ \ \text{on $\partial K$}
\end{cases}
\end{align*}
with $f \in L^2(\tilde \Omega)$ being any function that satisfies the compatibility condition
$$
\int_{K} f = \int_{\partial K} A \cdot \nu \stackrel{\eqref{same.flux}}{=} \int_{\partial K} \tilde A \cdot \nu.
$$ 
\end{proof}
\subsection{Properties of the surfaces $\nu_l(\cdot, \cdot)$.}
In this whole section we assume that $b_\eps$ is as in \eqref{def.b.eps}. 
\begin{lem}\label{l.derivatives.nu}
For every $n \in \N$, let $\nu_n: \T \times \R \to \R$ be the surfaces defined in \eqref{def.nu}. For $n\in \N$ and $(\xi, k) \in \T \times \R$, let $H_n(\xi, k, \cdot)$ be the eigenfunction for $\mathcal{O}(k,\xi)$ in \eqref{harmonic.oscillator.general} associated to the eigenvalue $\nu_n(\xi, k)$. Then the function $\nu_n$ is differentiable and
\begin{equation}\label{derivatives.nu}
\begin{aligned}
\partial_\xi \nu(\xi , k) &= 2 \int_\R (\int_0^x b(\xi, t) \, \d t - k) |H_n(\xi, k, x)|^2 \, \d x,\\
\partial_k \nu(\xi, k) &= 2 \int_\R (\int_0^x \partial_\xi b(\xi, t) \, \d t) (\int_0^x b(\xi, t) \, \d t - k) |H_n(\xi, k, x)|^2 \, \d x.
\end{aligned}
\end{equation}
\end{lem}

\begin{lem}\label{l.implicit.function}
Let $\lambda \notin \Sigma$, with $\Sigma$ as in \eqref{sigma.set}. Then, there exist $N \in \N$ functions $\{k_{j}\}_{j=1}^N \subset C^2( \T)$ satisfying for every $j =1, \cdots, N$
\begin{align}\label{eikonal.equation}
\nu_{n_j}( \xi, k_{j}(\xi)) = \lambda \ \ \ \ \text{for some $n_j \in \N$ and for all $\xi \in \T$}.
\end{align}  
Furthermore, there exists $\delta > 0$ such that for every $\xi \in \T$ it holds that
\begin{align}\label{distance.solutions}
|k_i(\xi) - k_j(\xi)| \geq \delta \ \ \ \ \text{for every $i, j = 1, \cdots, N$, such that $i\neq j$.}
\end{align} 

\smallskip

Finally, for each $i=1, \cdots, N$, the derivative $k_i'$ admits the representation
\begin{equation}\label{derivative.k}
\begin{aligned}
k_i'(\xi)= \biggl( \int_\R (k(\xi) - \int_0^x b(\xi, t) \d t) (\int_0^x& \partial_\xi b(\xi, t) \d t) |H_i(\xi, k_i(\xi), x)|^2 \d x \biggr)^{-1}\\
&\times \biggl(  \int_\R (k(\xi) - \int_0^x b(\xi, t) \d t) |H_i(\xi, k_i(\xi), x)|^2 \d x  \biggr). 
\end{aligned}
\end{equation}
and, for every $\xi \in \T$, the function $F(\cdot):= \frac{d}{d\xi} H_n(\xi, k_n(\xi), \cdot)$ is the unique solution to
\begin{equation}
\begin{aligned}\label{deriv.H}
-\partial_x^2 F(x) & + (\int_0^x b(\xi, t)\, \d t - k_n(\xi))^2 F(x)\\
& = \lambda F(x) - 2 (\int_0^x b(\xi, t) \, \d t - k_n(\xi))(\int_0^x \partial_\xi b(\xi, t) \, \d t + k'_n(\xi)) H_n(\xi, k_n(\xi), x) \ \ \ \text{in $\R$}\\
\end{aligned}
\end{equation}
and such that
\begin{equation}\label{orthogonality}
\int_\R F(x) H_n(\xi, k_n(\xi), x) \, \d x = 0.
\end{equation}
 \end{lem}

\begin{proof}[Proof of Lemma \ref{l.derivatives.nu}]
{The proof of this statement follows by the same argument used for the Feynman-Hellmann formula (e.g. \cite[Subsection 3.2.2]{HelfferFournais}): We differentiate the spectral problem solved by $H_n(\xi, k, \cdot)$ in $\xi$ or $k$ and test the resulting equation with $H_n(\xi, k, \cdot)$ itself. We stress that, since by construction $\| H_n(\xi, k , \cdot) \|_{L^2(\R)} =1$ for every $ \xi \in \T$ and $k\in \R$, it follows that both $\partial_k H_n(\xi, k ,\cdot)$ and $\partial_\xi H_n(\xi, k,\cdot)$ are orthogonal with respect to $H_n(\xi, k, \cdot)$.}
\end{proof}

\smallskip

\begin{proof}[Proof of Lemma \ref{l.implicit.function}]
The proof of \eqref{eikonal.equation} and \eqref{distance.solutions} and is an easy consequence of the assumption that $\lambda \notin \Sigma$ and the Implicit Function Theorem. Let $\lambda \notin \Sigma$. Then, if for every $n\in \N$ the equation $\nu_n (k, \xi) = \lambda$ does not admit a solution, the lemma is satisfied with the choice $N=0$. Let us assume, instead, that there exists $(\xi_1, k_1), \cdots (\xi_N, k_N) \in \T \times \R$ for $N \geq 1$ such that, for every $j=1, \cdots, N$ there is $n_j \in \N$ such that $\nu_{n_j}(\xi_j, k_j) = \lambda$. We stress that, by Lemma \ref{l.branches}, (i), it follows that the numbers $n_j \in \N$ satisfy a uniform upper bound that depends on $|\lambda|$.

\smallskip

Since $\lambda \notin \Sigma$ and $N$ is finite, we appeal to the Implicit Function Theorem and infer that there exist $\kappa > 0$ and functions $\{ k_j \}_{j=1}^N$ such that $k_j(\xi_j)=k_j$ and $\nu_{n_j}(\xi, k_j(\xi))=\lambda$ for every $|\xi - \xi_j| < \kappa$. {The regularity of each $\nu_n$ (Lemma \ref{l.branches}) and the uniform upper bound for $n_j$ yields that the size $\kappa$ is uniform in $\xi \in \T$. This, together with the fact that, by assumption $\partial_k \nu_{n_j}(\xi, k) \neq 0$ for every $\xi \in \T$ and $k\in \R$ such that $\nu_{n_j}(\xi, k) = \lambda$ implies that the same argument may be extended to a globally defined function $k_j : \T \to \R$.}

\smallskip

Property \eqref{distance.solutions} may be proven by contradiction: Since $\T$ is compact and $N\in \N$ is finite, it follows that if \eqref{distance.solutions} does not hold, we may then find $i, j=1,\cdots, N$ with $i \neq j$ and a value $\xi_* \in \T$ such that $k_*:= k_j(\xi_*) = k_i(\xi)$.  By definition, we know that $\nu_{n_i}(\xi, k_i(\xi)) = \nu_{n_j}(\xi, k_j(\xi))= \lambda$. Since for $\xi \in \T$ and $k_* \in \R$ fixed the sequence $\{ \nu_j (\xi, k_*) \}_{j \in \N}$ is strictly monotone in $j \in \N$ (c.f. Lemma \ref{l.branches}), it follows that necessarily $n_i = n_j:= n$. On the other hand, since $\lambda \notin \Sigma$, the implicit function theorem yields that $k_j(\xi) = k_i(\xi)$ for $|\xi - \xi_*| < \kappa$. The uniformity of $\kappa$ over $\T$, implies that the same argument may be extended o the full circle $\T$. This implies $k_j \equiv k_i$ and yields a contradiction since the two curves are assumed to be distinct.

\smallskip

We now turn to \eqref{derivative.k}: This is an immediate consequence of the definition of the curves $\{ k_j \}_{j\in \N}$ via the Implicit Function Theorem and the representation for the derivatives $\partial_k \nu_l, \partial_\xi \nu_l$ of Lemma \ref{l.derivatives.nu}. Similarly, \eqref{deriv.H} and \eqref{orthogonality} follow by differentiating in $\xi$ the equation 
\begin{align}
\bigl(-\partial_x^2 + (\int_0^x b(\xi, s) \d s - k_j(\xi)\bigr)^2 H_{n_j}(\xi, k_j(\xi), x) = \lambda H_{n_j}(\xi, k_j(\xi), x) \ \ \ \text{$x \in \R$ and for every $ \xi \in \T$.}
\end{align}
and the condition $\| H_n(\xi, k(\xi) , \cdot) \|_{L^2(\R)} =1$ for every $ \xi \in \T$. We stress that \eqref{orthogonality} implies that the solution to \eqref{deriv.H} is unique.
\end{proof}


\begin{thebibliography}{9}
\bibitem{AvronHerbstSimon}
J. Avron, I. Herbst and B. Simon,
\newblock Schr\"odinger operators with magnetic fields. I. General interactions,
\newblock {\em Duke Math. J.} (1978), 4:847 - 883.

\bibitem{barbaroux.ecc}
J.-M. Barbaroux, L. Le Treust, and N. Raymond, and E. Stockmeyer.  
\newblock On the semi-classical spectrum of the Dirichlet-Pauli operator.
\newblock{\em ArXiv Preprint} (2020).

\bibitem{Bernevig.book}
B. A. Bernevig, and T. L. Hughes.
\newblock{Topological Insulators and Topological Superconductors}.
\newblock {\em Princeton University Press} (2013).

\bibitem{Dombrowski.Hislop}
N. Dombrowski,  P. D. Hislop, and E. Soccorsi,  \textit{Edge Currents and Eigenvalue Estimates for Magnetic Barrier Schr\"odinger Operators},  Asymptotic Analysis (2014), 89({\bf 3-4}): 331-363.

\bibitem{Dombrowski.Raikov}
N. Dombrowski, F. Germinet, and G. Raikov, \textit{Quantization of the edge conductance for magnetic perturbation of Iwatsuka Hamiltonians}, Ann. H. Poincar\'e (2011), 12: 1169 -1197.

\bibitem{Gilbarg.Trudinger}
D. Gilbarg, and N. S. Trudinger, \textit{Elliptic Partial Differential Equations of Second Order}, Classics in Mathematics (2001), Vol. 224, Springer-Verlag Berlin Heidelberg.

\bibitem{GV}
A. Giunti, and J.J.L. Vel\'azquez, \textit{Edge states for the magnetic Laplacian in domains with smooth boundary}, SIAM J. Math. Anal. (2021), 53({\bf 3}): 3602 - 3643.

\bibitem{Grafakos.book}
L. Grafakos, \textit{Classical Fourier Analysis}, Graduate Texts in Mathematics (2008), Vol. 249, Springer-Verlag New York. 

\bibitem{Hall.book}
B. C. Hall, \textit{Quantum theory for mathematicians}, Graduate Texts in Mathematics (2013), Vol. 267, Springer-Verlag New York.

\bibitem{HelfferFournais}
{S. Fournais, and B. Helffer}.
\newblock{Spectral Methods in Surface Superconductivity},
\newblock {\em Progress in Nonlinear Differential Equations and Their Applications} (2010), Vol. 77, Birkh\"auser Basel.

\bibitem{Hislop.Soccorsi}
P. D. Hislop ,and E. Soccorsi, \textit{Edge states induced by Iwatsuka Hamiltonians with positive magnetic fields}, {Journal of Mathematical Analysis and Applications} (2015), {422}({\bf 1}): 594-624.

\bibitem{Hornberger.Smilansky}
K. Hornberger, and U. Smilansky, \textit{Magnetic edge states}, Phys. Reports (2002), 367, 249-385

\bibitem{Iwatsuka}
A. Iwatsuka,  \textit{Examples of Absolutely Continuous Schr\"odinger Operators in Magnetic Fields}, {Publ. RIMS}, Kyoto Univ. (1985), 21({\bf 2}): 385-401.

\bibitem{Miranda.Popov}
P. Miranda, and N. Popoff, \textit{Spectrum of the Iwatsuka Hamiltonian at thresholds}, Journal of Mathematical Analysis and Applications (2017), {460}: 516-545.

\bibitem{Raymond_book}
N. Raymond,
\newblock{Bound States of the Magnetic Schr\"odinger Operator},
\newblock {\em  EMS Tracts in Mathematics}(2017), Vol. 27.

 \bibitem{ReedSimon}
M. Reed and B. Simon.
\newblock Methods of Modern Mathematical Physics. Vol.IV: Analysis of operators.
\newblock {\em Academic Press} (1978).

\bibitem{ReijniersPeeters}
J. Reijniers, and F. M. Peeters, \textit{Snake orbits and related magnetic edge states}, J. Phys. Condens. Matt. (2000), {\bf 12}.

\bibitem{RozTas} 
G. Rozenblum, and G. Tashiyan, \textit{On the spectral properties of the Landau Hamiltonian perturbed by a moderately decaying magnetic field}, Rev. Math. Phys. (2009), 32({\bf 4}): 169-186.

\end{thebibliography}
\end{document}